\documentclass[11pt]{article}
\usepackage{amsthm}
\usepackage{amssymb}
\usepackage{epsfig}
\usepackage{graphicx}
\usepackage{epstopdf}
\usepackage{amsthm}
\usepackage{srcltx}
\usepackage[dvipsnames,usenames]{color}
\usepackage{paralist}
\usepackage{amsmath}
\usepackage{amsfonts}

\newcounter{contador}
\newtheorem{propo}[contador]{Proposition}
\newtheorem{teo}[contador]{Theorem}
\newtheorem{lem}[contador]{Lemma}

\newcommand{\e}{\mathrm{e}}

\newcommand{\R}{{\mathbb R}}

\newcommand{\N}{{\mathbb N}}

\title{Periodic orbits of discrete and continuous dynamical systems via Poincar\'{e}-Miranda theorem
\footnote{{\bf Acknowledgements.}
The authors are supported by
Ministry of Economy, Industry and Competitiveness--State Research Agency
of the Spanish
Government through grants MTM2016-77278-P  (MINECO/AEI/FEDER, UE, first
author) and DPI2016-77407-P
 (MINECO/AEI/FEDER, UE, second author). The first author is also supported by the grant 2017-SGR-1617  from
AGAUR,  Generalitat de Catalunya. The second author acknowledges the
group's research recognition 2017-SGR-388 from AGAUR, Generalitat de
Catalunya. This work was completed at the Erwin Schr\"odinger
International Institute for  Mathematics and Physics, when the
authors participated in the ESI Research in Teams project 2018.}}

\author{Armengol Gasull$^{(1)}$ and V\'{\i}ctor Ma\~{n}osa $^{(2)}$
  \\*[.1truecm]
{\small \textsl{$^{(1)}$ Departament de Matem\`{a}tiques}}
\\*[-.25truecm] {\small \textsl{Universitat Aut\`{o}noma de Barcelona,}}
\\*[-.25truecm] {\small \textsl{08193 Bellaterra, Barcelona, Spain}}
\\*[-.25truecm] {\small \textsl{gasull@mat.uab.cat}}
\\*[-.25truecm] {\small \textsl{$^{(2)}$ Departament de Matem\`{a}tiques,}}
\\*[-.25truecm] {\small \textsl{Universitat Polit\`{e}cnica de Catalunya}}
\\*[-.25truecm] {\small \textsl{Colom 11, 08222 Terrassa, Spain}}
\\*[-.25truecm] {\small \textsl{victor.manosa@upc.edu}}}

\begin{document}

\maketitle
\begin{abstract}

We present a systematic methodology to determine and
locate analytically
 isolated periodic points of discrete and continuous dynamical systems with
 algebraic nature. We apply this method to a wide range of examples, including
  a one-parameter family  of counterexamples to the discrete Markus-Yamabe conjecture
  (La Salle conjecture); the study of the low periods of a Lotka-Volterra-type map;
  the existence of three limit cycles for a piece-wise linear planar vector
  field; a new counterexample of Kouchnirenko's conjecture;
  and an alternative proof of the existence of a class of symmetric central configuration of the $(1+4)$-body problem.

\end{abstract}



\section{Introduction and main results}

Periodic orbits are one of the main objects of study of the theory
of  dynamic systems. A priori there are many ways to prove the
existence periodic orbits, for instance one can try to apply the
plenty of available fixed point theorems \cite{GD} or results
guaranteeing the existence of zeros, since  periodic orbits are
always  solutions of equations of the form
$G(\mathbf{x})=\mathbf{x}$, where $G$ is a return map in the
continuous case, and $G=F^p$ for some $p\in\N$ in the case of a
discrete system given by a map $F$. However when one tries to apply
these results to a particular case it is not always easy to find
effective ways to check the hypotheses.  An example of this fact
appears when trying to use the Newton-Kantorovich Theorem \cite{IK}.
By using this approach, some bounds of the partial derivatives of
the involved functions must be obtained. The work done in \cite{BL}
exemplifies clearly the difficulties of this approach.

In this work we present an effective procedure to prove the
existence,  determine the number and locate periodic orbits of
dynamical systems of both discrete and continuous nature.
This procedure is explained in detail in the
 next sections. As we will see, one of the main features of this
 procedure is the use of the Poincar\'{e}-Miranda theorem (PMT for short).
We believe that one of the advantages of using PMT for finding fixed
points of a given function is that only the signs of the components
of it have to be controlled on some suitable sets, which is
straightforward in the case that either the equations are polynomial
or the problem can be polynomialized (see for instance the proof of
Theorem \ref{T-3LC} in Section \ref{S-Piecewise}). Recall that the
use of Sturm sequences for polynomials in $\mathbb{Q}[x]$ allows to
control their signs on intervals with rational endpoints
(\cite{StB}).

The PMT is the extension of the Bolzano theorem to higher
 dimensions. It was formulated and proved by H.~Poincar\'e in 1883 and 1886
 respectively, \cite{Poinc1,Poinc3}. C.~Miranda
re-obtained the result as an equivalent formulation of Brouwer fixed
point theorem  in 1940,  \cite{Miranda}. Recent proofs are presented
in \cite{K,V}.  For completeness, we recall it. As usual, $\overline
S$ and $\partial S$ denote, respectively, the closure and the
boundary of a set $S\subset\R^n.$

\begin{teo}[Poincar\'e-Miranda]\label{T-PM-n}
       Set $\mathcal{B}=\{\mathbf{x}=(x_1,\ldots,x_n)\in\R^n\,:\,L_i<x_i<U_i, 1\leq i\leq n\}$. Suppose that
         $f=(f_1,f_2,\ldots,f_n):\overline{\mathcal{B}}\rightarrow R^n$ is
         continuous,  $f(\mathbf{x})\neq\mathbf{0}$
           for all $\mathbf{x}\in\partial \mathcal{B}$, and for  $1\leq i\leq
       n,$
       $$f_i(x_1,\ldots,x_{i-1},L_i,x_{i+1},\ldots,x_n)\leq 0\, \mbox{ and }\, f_i(x_1,\ldots,x_{i-1},U_i,x_{i+1},\ldots,x_n)\geq 0,$$
       Then, there exists $\mathbf{s}\in\mathcal{B}$ such that  $f(\mathbf{s})= \mathbf{0}$.
       \end{teo}

For short, when given a map $f$ we have a box $\mathcal{B}$ such
that the hypotheses of the PMT hold we will say that $\mathcal{B}$
is a PM box. When we try to apply PMT to some $f,$ sometimes  it is
better to consider some permutation of its components.

The paper is structured as follows: we start giving a new degree 6
counterexample of Kouchnirenko conjecture to illustrate the use and
utility of our approach. In Section \ref{S-Markus Yamabe},  we prove
the existence of a 1-parameter family of rational counterexamples to
a conjecture of La Salle (also known as discrete Markus-Yamabe
conjecture) that extends the results of \cite{CGM14} providing also
an alternative  proof of them. In Section~\ref{S-Thue-Morse} we
prove the existence of exactly two $5$-periodic orbits and three
$6$-periodic orbits in a certain region for a Lotka-Volterra-type
map correcting and complementing some results that appear in the
literature. In Section~\ref{S-Piecewise} we provide another example
of planar piecewise linear differential system with two zones having
$3$-limit cycles.  Finally, in Section \ref{S-central} we use PMT to
give an alternative proof of the existence of a type of symmetric
central configuration of the $(1+4)$-body problem.

\section{A new counterexample to  Kouchnirenko
conjecture}

Descartes' rule asserts that a 1-variable real polynomial with $m$
monomials has at most $m-1$ simple  positive real roots. The
Kouchnirenko conjecture  was posed as an attempt to extend this rule
to the several variables context.  In the 2-variables case this
conjecture said that \emph{a real polynomial system
$f_1(x,y)=f_2(x,y)=0$ would have at most $(m_1-1)(m_2-1)$ simple
solutions with positive coordinates, where $m_i$ is the number of
monomials of each $f_i$}.  This conjecture was stated by
A.~Kouchnirenko in the late 70's, and published in the
A.~G.~Khovanski\u{\i}'s paper \cite{Kh}. In 2000, B.~Haas (\cite{H})
constructed a family of counterexamples given by two trimonomials,
being the  minimal degree of these counterexamples 106. In 2007 a
much simpler family of counterexamples was presented in \cite{DRRS},
being the simplest one again formed by two trimonomials, but of
degree $6.$ Both examples have exactly $5$ simple solutions with positive
coordinates instead of the $4$ predicted by the conjecture. In 2003,
it was proved in \cite{LRW} that any pair of bivariate trinomials
has at most 5 simple solutions.

We will prove in a very simple way, by using PMT, that system
\begin{equation*}
\left\{\begin{array}{l}
P(x,y):=x^{6}+a y^3-y=0,\\
Q(x,y):=y^{6}+a x^3-x=0,
\end{array}\right.
\end{equation*}
with $a=61/43\simeq 1.41860465$ is  a counterexample of the
conjecture. We remark that in \cite{DRRS} it was given the
counterexample with $a=44/31.$ The reason why we have changed this
parameter is that it can be proved that when $a=\overline
a=7\times12^{4/5}/36\simeq 1.14195168$ the above system has the
multiple solution $(s,s)$ with $s=12^{3/5}/6\simeq 0.74021434$ and
$\overline a$ is quite close to $44/31\simeq 1.4193548,$ making
that, for that system,  3 of the its 5 solutions with positive
entries are very close to each other. By using the approach
developed in \cite{GGG,GGG2}, or the tools of \cite{DRRS}, it can be
proved that counterexamples to the conjecture only appear for $a\in
(\underline a, \overline a),$ where $\underline a\simeq 1.4176595.$
Both values are zeroes of some irreducible factor of the polynomial
$\Delta_y(\mathrm{Res}(P,Q;x)),$ where $\Delta_y$ and $\mathrm{Res}$
denote, as usual, the discriminant and the resultant respectively.
Hence, our value of $a$ has also small numerator and denominator and, moreover, it is near the middle point of this interval, making that in the computations of
our proof the rational numbers  involved are simpler that the ones
needed to use our approach when $a=44/31.$ We prove:

\begin{propo} The bivariate trinomial system
\begin{equation}\label{E-sistema-nou}
\left\{\begin{array}{l}
P(x,y)=x^{6}+\frac{61}{43} y^3-y=0,\\
Q(x,y)=y^{6}+\frac{61}{43} x^3-x=0,
\end{array}\right.
\end{equation}
has 5 real simple solutions with positive entries.
\end{propo}

\begin{proof}
It is not difficult to find numerically 5 approximated solutions of
the system. They are $(\widetilde x_1,\widetilde x_5)$, $(\widetilde
x_2,\widetilde x_4)$, $(\widetilde x_3,\widetilde x_3)$,
$(\widetilde x_4,\widetilde x_2)$, $(\widetilde x_5,\widetilde
x_1)$, where $\widetilde  x_1=0.59679166,$ $\widetilde x_2=
0.68913517,$  $\widetilde x_3= 0.74035310,$ $\widetilde x_4=
0.77980435$ and  $\widetilde x_5= 0.81602099.$
 We consider the following 5 intervals, with $\widetilde x_i\in I_i,$
\begin{align*}
I_1&=\left[{\frac{1}{2}},{\frac{1619}{2500}}\right],\,
I_2=\left[{\frac{1619}{2500}},{\frac{18}{25}}\right],\,
I_3=\left[{\frac{18}{25}},{\frac{75857}{100000}}\right],\\
I_4&=\left[{\frac{75857}{100000}},{\frac{4}{5}}\right],\,
I_5=\left[{\frac{4}{5}},{\frac{83}{100}}\right].
\end{align*}
Let us prove that system \eqref{E-sistema-nou} has 5 actual
solutions $( x_1, x_5)$, $( x_2, x_4)$, $( x_3, x_3)$, $( x_4,
x_2)$, $( x_5, x_1)$,   with $x_i\in I_i.$ Firstly, since $
P(x,x)=Q(x,x)=x^{6}+{61}x^3/43-x,$ by Descartes rule we know that
there is exactly one simple positive real root of $P(x,x).$ By
Bolzano theorem it belongs to $I_3.$ So there is a solution
$(x_3,x_3)$ of the system in $I_3\times I_3$.

By the symmetry of the system, if $(x^*,y^*)$ is one of its
solutions then $(y^*,x^*)$ also is. Hence,
 we only need to prove that there are two suitable different
 solutions. This will be proved by applying the PMT to the boxes
$I_1 \times I_5,$ and $I_2 \times I_4,$ which are depicted in
Figure~\ref{F-kouch12}.

\begin{figure}[h]
\centerline{\includegraphics[scale=0.30]{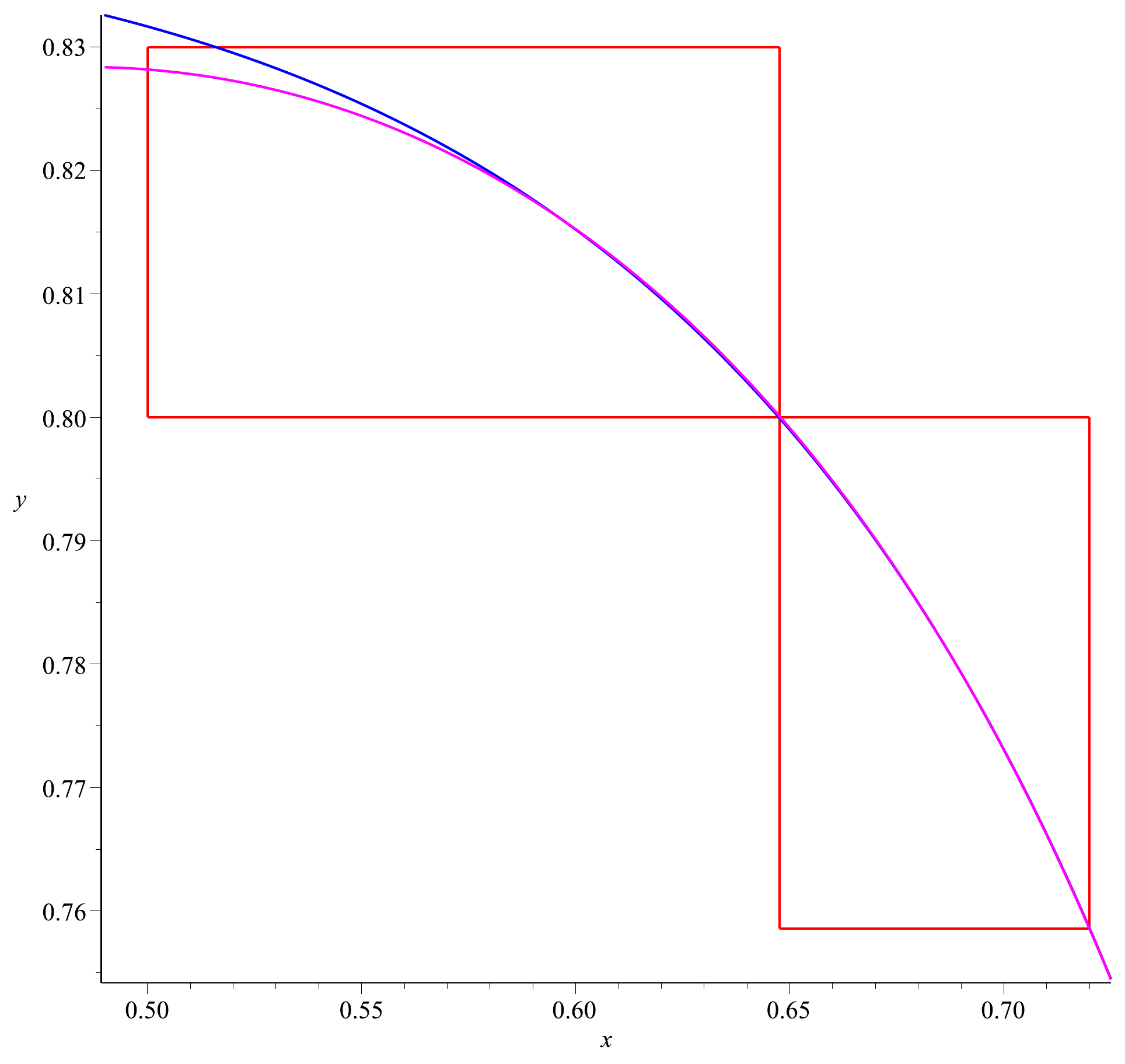}}
\caption{Intersection of the curves $P(x,y)=0$ (in blue) and $Q(x,y)=0$
(in magenta). The PM boxes $I_1\times I_5$, $I_2\times I_4$ (left
and right respectively, in red).} \label{F-kouch12}
\end{figure}

We start  applying the PMT to the box $I_1\times I_5.$ Consider the
polynomials
$$
P\left(\frac12,y\right)=\frac {61}{
43}\,{y}^{3}-y+\frac{1}{2^6}\quad\mbox{and}\quad P\left(\frac
{1619}{2500},y\right)=\frac {61}{
43}\,{y}^{3}-y+\left(\frac{1619}{2500}\right)^6.
$$
By computing their corresponding Sturm sequences we get that both
have no roots in $[4/5,83/100].$ Moreover $P(1/2,y)<0$ and
$P({1619}/{2500},y)>0$ on this interval. Similarly we get that
$$
Q\left(x,\frac45\right)={\frac {61}{43}}\,{x}^{3}-x+\left(\frac{
4}{5}\right)^6<0\quad\mbox{and}\quad
Q\left(x,\frac{83}{100}\right)={\frac
{61}{43}}\,{x}^{3}-x+\left(\frac{83}{100}\right)^6>0
$$
on  $[1/2,{1619}/{2500}]$. Hence, $I_1\times I_5$ is under the
hypotheses of the PMT, and system~\eqref{E-sistema-nou} has a
solution $(x_1,x_5)$ in this box.

 By using the same arguments one gets that the box
$I_2\times I_4,$  contains another solution $(x_2,x_4)$ of our
system. In this case the polynomials involved are even simpler. In
this occasion, for $y\in [75857/100000,4/5]$ it holds that
\[
Q\left(\frac{1619}{2500},y\right)=y^6-\frac{176243010801}{671875\times10^6}<0\quad\mbox{and}\quad
Q\left(\frac{18}{25},y\right)=y^6-\frac{127998}{671875}>0,
\]
and for $x\in[1619/2500,18/25],$ that
\[
P\left(x,\frac{75857}{10^5}\right)=x^6-\frac{5991841917684627}{43\times
10^{15}}<0\,\,\mbox{and}\,\, P\left(x,\frac4
5\right)=x^6-\frac{810993}{43\times 10^6}>0.
\]

 The above facts prove that in the boxes $I_1\times I_5$,
$I_2\times I_4$ and their symmetric ones, $I_5\times I_1$ and
$I_4\times I_2$,
 there are at least $4$ solutions of the studied system.  These solutions together with the
   solution in the diagonal give the $5$ announced  solutions with positive coordinates. To prove they are simple solutions
    we first compute
$$J(x,y):=\det{\rm D}(P,Q)=
36 {x}^{5}{y}^{5}-{\frac {33489 }{1849}}{x}^{2}{y}^{2}+{\frac {183
}{43}}{x}^{2}+{\frac {183 }{43}}{y}^{2}-1.
$$
Since $\mathrm{Res}(\mathrm{Res}(P,Q;x),\mathrm{Res}(P,J;x);y)\neq
0$,  $J$ does  not vanish on the solutions (real or complex) of
system \eqref{E-sistema-nou}. Hence all their solutions  are simple.
In fact, by using that a bivariate
  trinomial system hay at most five different solutions (\cite{LRW}) or the tools of
  the so-called \emph{discard procedure}, that we will introduce in Section \ref{S-Thue-Morse} we get
  that 5 is the exact number of solutions with positive entries
  and that these solutions together with $(0,0)$ are the only real solutions of the system.
\end{proof}

\section{A counterexample to the discrete Markus-Yamabe conjecture revisited}\label{S-Markus Yamabe}

In \cite{LS},  J.~P.~La Salle  proposed some possible sufficient
conditions for  discrete dynamical systems with a fixed point, $
x_{n+1}=F(x_n),\,x\in \mathbb{R}^n, $ to be globally asymptotically
stable (GAS). One of these  conditions is:
\begin{equation}\label{E-cond-MY-I}
\mbox{ For all }x\in \mathbb{R}^n,\,
\rho\left(\mathrm{D}F(x)\right)<1,
\end{equation}
where $\rho$ is the spectral radius of the differential matrix. This
condition  is known as a \emph{discrete Markus-Yamabe-type}
condition because of its similarity with the conditions of
Markus-Yamabe conjecture for ordinary differential equations, stated
by L.~Markus and H.~Yamabe in 1960 \cite{MY}, that has been proved
to be true in dimension two and false in superior dimensions, see
for instance \cite{CEGHM,G}.

In \cite{CGM14} the authors consider rational maps of the form
\begin{equation}\label{E-MY}
F(x,y)=\left(y,-bx+\frac{a}{(1+y^2)^2}\right),
\end{equation}  and  prove that \emph{there exist some real values, $a=a^*$ and $b=b^*,$ such that
the map \eqref{E-MY} satisfies the Markus-Yamabe condition
\eqref{E-cond-MY-I} and it has  the  3-periodic point
$(-0.1,0.25).$} Moreover they show \emph{numerically} that for
$a^*=1.8$ and $b^*=0.9$ a $3$-periodic orbit seems to exist.  This
example  was proposed to simplify the previous one given by
W.~Szlenk, see \cite[Appendix]{CGM99}; and to show that even for
systems coming from rational difference equations the discrete
Markus-Yamabe conjecture does not hold.

In this section we apply the PMT to give a  simple proof of the
following result, that in particular fixes the numerical
counterexample presented in \cite{CGM14}.

\begin{propo}\label{P-MY2}
For $b\in \mathbf{B}:=[113/128,2916/3125]\simeq[0.883,0.933]$ the
map
\begin{equation*}
F(x,y;b)=\left(y,-bx+\frac{2b}{(1+y^2)^2}\right),
\end{equation*}
satisfies the Markus-Yamabe condition \eqref{E-cond-MY-I} and  has a
$3$-periodic orbit.
\end{propo}

Prior to prove this proposition, we recall the following auxiliary
lemma, that is a simplified version of a result given in \cite{GGG}.

  \begin{lem}\label{L-zeros}
   Let $G(x;b)=g_n(b) x^n+g_{n-1}(b) x^{n-1}+\cdots+g_1(b)
   x+g_0(b)$ be
   a family of real polynomials that depend continuously on one real parameter
 $b\in\mathbf{B}=[b_1,b_2]\subset\mathbb{R}$. Fix $J=[\underline {x},\overline{x}]\subset\R$ and assume that:
   \begin{enumerate}[(i)]
    \item There exists $b_0\in \mathbf{B}$ such that  $G(x;b_0)$ has no  real roots in  $J$.
    \item For all $b\in \mathbf{B}$,  $G(\underline {x};b)\cdot G(\overline {x};b)\cdot \Delta_x(G_b)\neq 0,$ where
    $\Delta_x(G(\cdot;b))$
    is the discriminant of $G(x;b)$ with respect to $x.$
  \end{enumerate}
   Then for all $b\in \mathbf{B}$, $G(x;b)$  has no  real roots in $ J$.
  \end{lem}

\begin{proof}[Proof of Proposition \ref{P-MY2}] We start noticing that in~\cite{CGM14}, it is proved
that the maps~\eqref{E-MY} satisfy condition \eqref{E-cond-MY-I} if
and only if $|a|<\sqrt{{11664}/{3125}}$ and $b\in({3125
a^2}/{11664},1).$  When $a=2b$ these conditions reduce to
$b\in\left(0,{2916}/{3125}\right).$

The 3-periodic points are solutions of system
$F^2(x,y;b)=F^{-1}(x,y;b)$, that can be studied trough the
equivalent system
\begin{equation}\label{E-F3i}
g_{i}(x,y;b):=\operatorname{Numer}\left(F^2_i(x,y;b)-F^{-1}_i(x,y)\right)=0,\quad
i=1,2,
\end{equation}
where, as usual, $G_i$ denotes the $i$-th component of a map $G$.
Some computations give
\begin{align*}
&g_{1}(x,y;b)=\!\!-{b}^{2}{x}^{5}{y}^{4}-2 {b}^{2}{x}^{5}{y}^{2}-2
{b}^{2}{x}^{3}{y}^{ 4}+{x}^{4}{y}^{5}-{b}^{2}{x}^{5}-4
{b}^{2}{x}^{3}{y}^{2}-{b}^{2}x{y}^
{4}+2 {x}^{4}{y}^{3}\\
&+2 {x}^{2}{y}^{5}+2 {b}^{2}{x}^{4}-2 {b}^{2}{x }^{3}-2
{b}^{2}x{y}^{2}-2 b{y}^{4}+{x}^{4}y+4 {x}^{2}{y}^{3}+{y}^{5
}+4 {b}^{2}{x}^{2}-{b}^{2}x-4 b{y}^{2}\\
&+2 {x}^{2}y+2 {y}^{3}+2 {b} ^{2}-2 b+y
\end{align*}
and that $g_2(x,y;b)$  has degree 21 in $(x,y)$ and degree $5$ in
$b.$ We omit its expression.

 We claim, now,  that for any
$b\in(\underline{b},\overline{b})=(113/128,123/128)\simeq(0.883,0.961)$
there is a solution of system \eqref{E-F3i} in the box
$\mathcal{B}:=\left[-0.2,0\right]\times\left[0,0.5\right]$,
corresponding with a $3$-periodic point, where this interval of
values of $b$ is not optimal. Notice that the box $\mathcal{B}$ does
not contain points on the diagonal line $y=x$ and so,  the found
solution is not a fixed point, but a $3$-periodic one. The curves
$g_1(x,y;0.9)=0$ (in blue) and $g_2(x,y;0.9)=0$ (in magenta),
together with the box $\mathcal{B},$ are depicted in Figure
\ref{F-MY-PM}.

\begin{figure}[h]
\centerline{\includegraphics[scale=0.25]{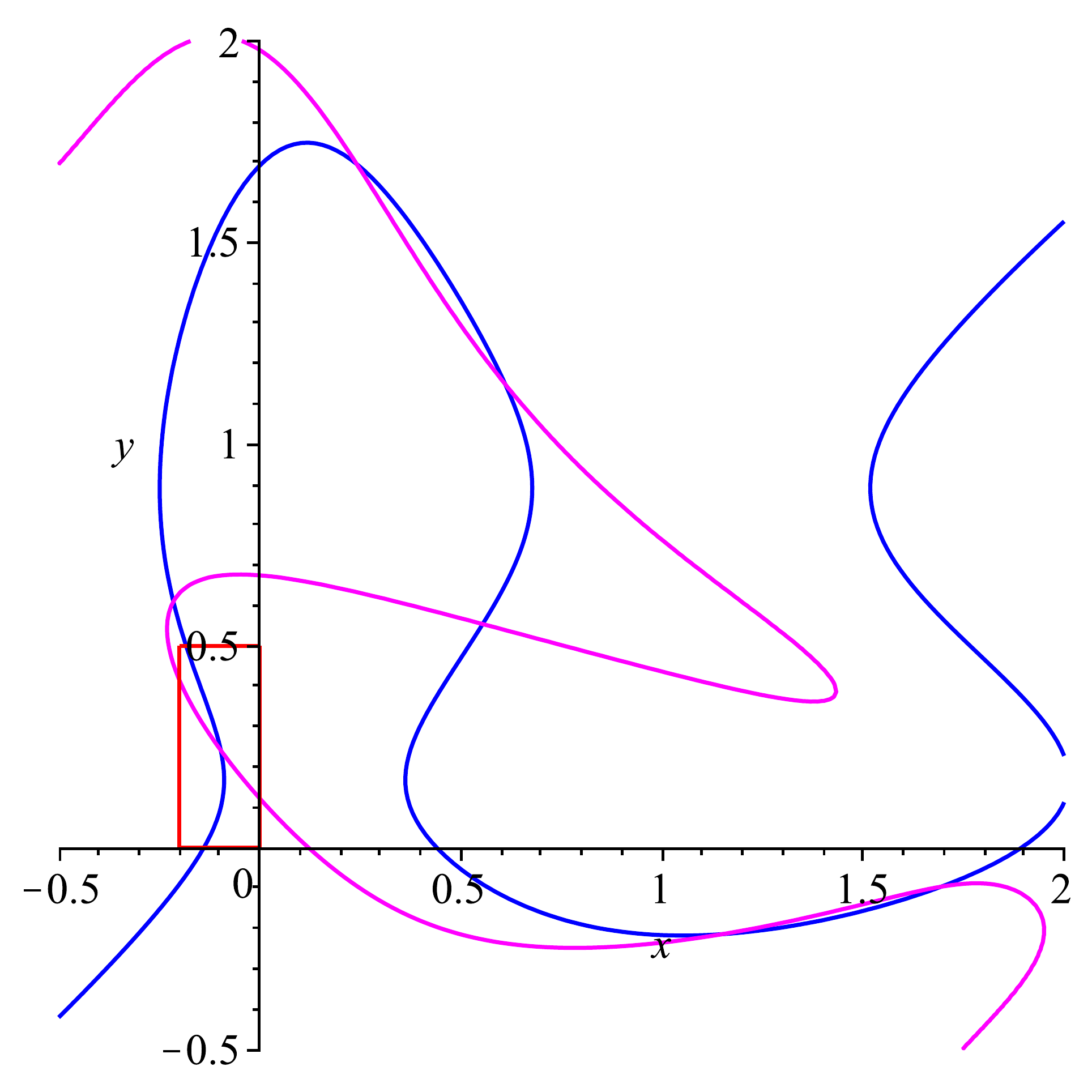}}
\caption{Intersection the curves $g_1(x,y;0.9)=0$ (in blue) and
$g_2(x,y;0.9)=0$  (in magenta).  It can be seen that there are seven
intersection corresponding to one fixed point and two different
$3$-periodic orbits. In red, a PM box of one of the solutions of
system~\eqref{E-F3i}.} \label{F-MY-PM}
\end{figure}

 In consequence, for $b\in\mathbf{B}$ the map satisfies the Markus-Yamabe condition \eqref{E-cond-MY-I}
and has a $3$-periodic point,  as we wanted to prove. The claim will follow from PMT applied to the map $f=(g_1,g_2)$, once we prove for all
$b\in\mathbf{B}$:
\begin{enumerate}[(I)]
\item $h_1(y;b):=g_1(-0.2,y;b)\cdot g_1(0,y;b)<0$ for $y\in[0,0.5],$
\item $h_2(x;b):=g_2(x,0;b)\cdot g_1(x,0.5;b)<0$ for $x\in[-0.2,0].$
\end{enumerate}

Items (I) and (II) will be consequences of Lemma \ref{L-zeros}. We
only give the details to prove item (I).

Condition (i) of the lemma holds taking $b^*=0.9,$ because
$$
g_1(-0.2,y;0.9)={\frac {676 }{625}}{y}^{5}-{\frac {126936
}{78125}}{y}^{4}+{\frac { 1352 }{625}}{y}^{3}-{\frac {253872
}{78125}}{y}^{2}+{\frac {676}{ 625}}y+{\frac{9954}{78125}}
$$
and $ g_1(0,y;0.9)={y}^{5}-\frac{9}{5}{y}^{4}+2{y}^{3}-{\frac
{18}{5}}{y}^{2}+y-{\frac{9}{50}} $ do not vanish in $[0,0.5],$ as
can be seen by computing their Strum sequences, and $h_1(0;0.9)<0.$

To check condition (ii) we prove that the polynomial in the variable
$b$, with rational coefficients and degree $71$,  $h_1(0;b)\cdot
h_{1}(0.5;b)\cdot \Delta_y(h_1(y,b))$, has no roots for
$b\in\mathbf{B}.$ This can be done again by computing its Sturm
sequence.~\end{proof}

By using, the approach introduced in next section  it is easy
 to prove for instance that the exact number of 3-periodic orbits of
the map given in  Proposition \ref{P-MY2} when $b=0.9$ is two, see
again Figure \ref{F-MY-PM}.

\section{Periodic orbits of a Lotka-Volterra map}\label{S-Thue-Morse}

We consider the following Lotka-Volterra type map
\begin{equation}\label{E-TMmap}
T(x,y)=\left(x(4-x-y),xy\right).
\end{equation}
The interest for this map has grown after its  consideration by
A.N.~Sharkovski\u{\i} \cite{S}. Notice that it unfolds the logistic
map. It appears in many applications (\cite{ED}), being one of the
most relevant ones, its relationship with some solutions of the
Schr\"odinger equations modeling 1-dimensional quasi-cristalls with
Thue-Morse sequence distributions, see \cite{AB}.

This map is typically studied in the triangle $\triangle\subset
\R^2$  with vertices $(0,0)$, $(4,0)$, $(0,4)$, which is invariant.
The low-period orbits of the map \eqref{E-TMmap} were studied in
\cite{BGLL,Mali}. It is known that in $\mathrm{Int}(\triangle)$ the
fixed point $(1,2)$  is unique; there are not $2$ and $3$-periodic
points;  there is a unique $4$-periodic orbit (which is explicitly
known, \cite{BGLL}); and that there are $5$ and $6$-periodic points.
The $5$-periodic orbit is claimed to be unique in \cite{BGLL}.  The
following result completes and corrects those obtained in the above
references.
\begin{teo}\label{T-TM}
The following statements hold:
\begin{enumerate}[(a)]
\item There exist exactly two different periodic orbits of minimal period $5$ of $T$ in $\mathrm{Int}(\triangle)$.
\item  There exist exactly three different periodic orbits of minimal period $6$ of $T$ in
$\mathrm{Int}(\triangle)$. Moreover one of them is
\[(u,1)\rightarrow (1,u) \rightarrow (3-u,u) \rightarrow (3-u,1)\rightarrow (1,3-u) \rightarrow (u,3-u) \rightarrow (u,1), \]
where $u=(3-\sqrt5)/2$ satisfies $u(3-u)=1.$
\end{enumerate}
\end{teo}

Notice that taking as $u=(3+\sqrt5)/2$ the other root of the same
polynomial we obtain the same orbit.

 To prove the above result we
use a methodology developed in \cite{GasLlorMan2018}, that can be
summarized as:

\begin{itemize}
\item We fix the period $p$. By using resultants, we  include the solutions of $
T^p(x,y)=(x,y),$ into the ones of an uncoupled system of equations
given by two 1-variable polynomials.

\item We use the corresponding Sturm sequences  for
isolating the
 real roots of each 1-variable polynomials, and we apply  a \emph{discard procedure} in order to remove those solutions
 of the later system that do not correspond with the periodic
 points.

 \item  We apply
the  PMT to prove that the non discarded
 solutions are actual solutions of the first system of polynomial equations.
\end{itemize}

\begin{proof}[Proof of Proposition \ref{T-TM}]
(a) We start noticing that imposing  $T^5(x,y)=(x,y)$, one has the
system of equations
\begin{equation}\label{E-TM-OP5}
x\cdot T_{5,1}(x,y)=0,\quad y\cdot T_{5,2}(x,y)=0,
\end{equation}
where $T_{5,1}$ and $T_{5,2}$ are polynomials with degree 31, and
263 and 222  monomials respectively. We consider the resultants of
these polynomials, and we remove the repeated factors and those
factors corresponding to $x=0$ and $y=0$.

{\footnotesize
\begin{align*}
&P(x):=\frac{\mathrm{Res}(T_{5,1},T_{5,2};y)}{{x}^{100}\left( x-2
\right)^{99}}=  \left( x-2 \right)\left( x-1 \right)  \left(
{x}^{5}-136 {x}^{4}+1784 {x}^{3}-5957 {x}^{2}+5850 x-1 \right)
\\
& \left( {x}^{10}-41 {x}^{9} +482 {x }^{8}-2624 {x}^{7}+7847
{x}^{6}-13837 {x}^{5}+14655 {x}^{4}-9088
{x}^{3}+3019 {x}^{2}-414 x\right.\\
&\left.+1 \right)  \left( {x}^{15}-178 {x}^{14}+ 7997
{x}^{13}-153777 {x}^{12}+1588330 {x}^{11}-9901048 {x}^{10}+
39727694 {x}^{9}\right.\\
&\left. -106108582 {x}^{8}+190846457 {x}^{7}-229400781 {x}
^{6}+179062441 {x}^{5}-85605963 {x}^{4}+22367351 {x}^{3}\right.
\\&\left.-2429213 {
x}^{2}+6279 x-1 \right),
\end{align*} }
\!\!and {\footnotesize
\begin{align*}
&Q(y):=\mathrm{Res}(T_{5,1},T_{5,2};x)=\left( y-2 \right)  \left(
{y}^{5}+14520 {y}^{4}+2662000 {y}^{3}+
121121000 {y}^{2}+878460000 y\right.\\
&   \left.+1464100000 \right) \left( {y}^{10}+ 594 {y}^{9}+16280
{y}^{8}+56320 {y}^{7}-567248 {y}^{6}+220704 {y}
^{5}+2656192 {y}^{4}\right.\\
&\left.-2725888 {y}^{3}-2385152 {y}^{2}+4088832 y-
1362944 \right)  \left( {y}^{15}+15156 {y}^{14}+11338084 {y}^{13}\right.\\
&+ 1961135256 {y}^{12}+120710774176 {y}^{11}+2862490382720 {y}^{10}+
25795669773184 {y}^{9}\\
&+52844703170304 {y}^{8}-280355579032320 {y}^{
7}-811324992569856 {y}^{6}+760407187850240 {y}^{5}\\
&\left.+2215201573881856
 {y}^{4}-1452783687979008 {y}^{3}-1660265095602176 {y}^{2}+
1449013276164096 y\right.\\& \left.-281389965541376 \right).
\end{align*}
}

By using the Sturm approach we obtain that $P(x)$ has 32 different
real roots (all of them positive), and $Q(x)$  has 31 different real
roots, 11 of them positive. Hence, each  solution in the positive
quadrant of system \eqref{E-TM-OP5} \emph{is  contained in
isolation} in one of the $352=32\times 11$ boxes $
\mathcal{I}_{i,j}=I_i\times J_j,$  $i=1,\cdots,32;$ $j=1,\ldots,11,$
where all $I_i$ and $J_j$ are intervals with positive rational
endpoints such that each one of them contains a positive root of $P$
and $Q$, respectively, in isolation.

As we have already explained, to discard those sets
$\mathcal{I}_{i,j}$ that do not contain any solution of system
\eqref{E-TM-OP5}, we apply the discard method presented in
\cite{GasLlorMan2018}.

We consider all boxes $\mathcal{I}_{i,j}$. For each one we want to
know whether   the function $f(x,y)=\sum_\ell M_\ell(x,y)$, where
$f$ can be either $T_{5,1}$ or $T_{5,2},$ has or not a fixed sign.

Setting $$ \mathcal{I}_{i,j}=[\underline{x},\overline{x}]\times
[\underline{y},\overline{y}]\subset(\R^+)^2,$$ for each monomial
$M_\ell(x,y)=a_\ell x^{\ell_1} y^{\ell_2}$
 one has $\underline{M}_{\ell}\le M(x,y) \le
  \overline{M}_{\ell}$, where  $\underline{M}_{\ell}
  =a_\ell\,\underline{x}^{\ell_1} \underline{y}^{\ell_2}$
 and $\overline{M}_\ell=a_\ell\,\overline{x}^{\ell_1} \overline{y}^{\ell_2}$ if  $a_\ell>0$, or
   $\underline{M}_{\,\ell}=a_\ell\,\overline{x}^{\ell_1} \overline{y}^{\ell_2}$
 and $\overline{M}_{\ell}=a_\ell\,\underline{x}^{\ell_1} \underline{y}^{\ell_2}$ if  $a_\ell<0$.

If either $0<\sum_\ell \underline{M}_{\,\ell}<\sum_\ell
        M_\ell(x,y)=f(x,y)$ or $f(x,y)=\sum_\ell M_\ell(x,y)<\sum_\ell
        \overline{M}_{\,\ell}<0$
then we can discard the box  $ \mathcal{I}_{i,j}$. If not, but we
suspect (by our previous numerical computations) that it should  be
discarded, we substitute it by one of smaller size.

To apply the discard procedure efficiently we need to compute the
intervals $I_i$ and $J_j$ with  maximum length $10^{-40}$ which are
given in the appendix.   It gives that each solution of system
\eqref{E-TM-OP5} must be contained in one of the following $11$
non-discarded boxes
\begin{equation}\label{E-caixesTM}
\begin{array}{llllllllll}
\mathcal{I}_{5,7},&\mathcal{I}_{6,11}, & \mathcal{I}_{7,9},&
\mathcal{I}_{8,6},&\mathcal{I}_{9,10},&
\mathcal{I}_{10,2},&\mathcal{I}_{14,3}, & \mathcal{I}_{20,1},& \mathcal{I}_{23,5},&\mathcal{I}_{24,4}\\
\end{array}
\end{equation}
and $\mathcal{I}_{11,8}=[1,1]\times[2,2]$ which, obviously
corresponds with the unique fixed  point of $T$ $(x,y)=(1,2)$, so we
discard it.

To prove that there is a (unique) solution of system
\eqref{E-TM-OP5} in each box, and  therefore there are 2-periodic
orbits with $10$ periodic points of minimal period $5$ we apply the
PMT. To illustrate the type of computations we deal with, we only
show one of the computations. We prove that there is a unique
solution in the box $\mathcal{I}_{9,10}$.

To obtain simpler expressions and work more comfortably we will show
that the  hypotheses of the PMT are verified for a bigger box
$\mathcal{B}=:=\left[0.6,1\right]\times\left[2.3,2.9\right],$ which
has been obtained by visual inspection, see Figure \ref{F-FigPM-TM},
instead of using the actual box. It is easy to check that the only
box of \eqref{E-caixesTM} contained in $\mathcal{B}$ is
$\mathcal{I}_ {9,10}$, and therefore if there is a solution of
system \eqref{E-TM-OP5} in $\mathcal{B}$ then it must be in
$\mathcal{I}_ {9,10}$, and be unique by construction.

\begin{figure}[h]
\centerline{\includegraphics[scale=0.28]{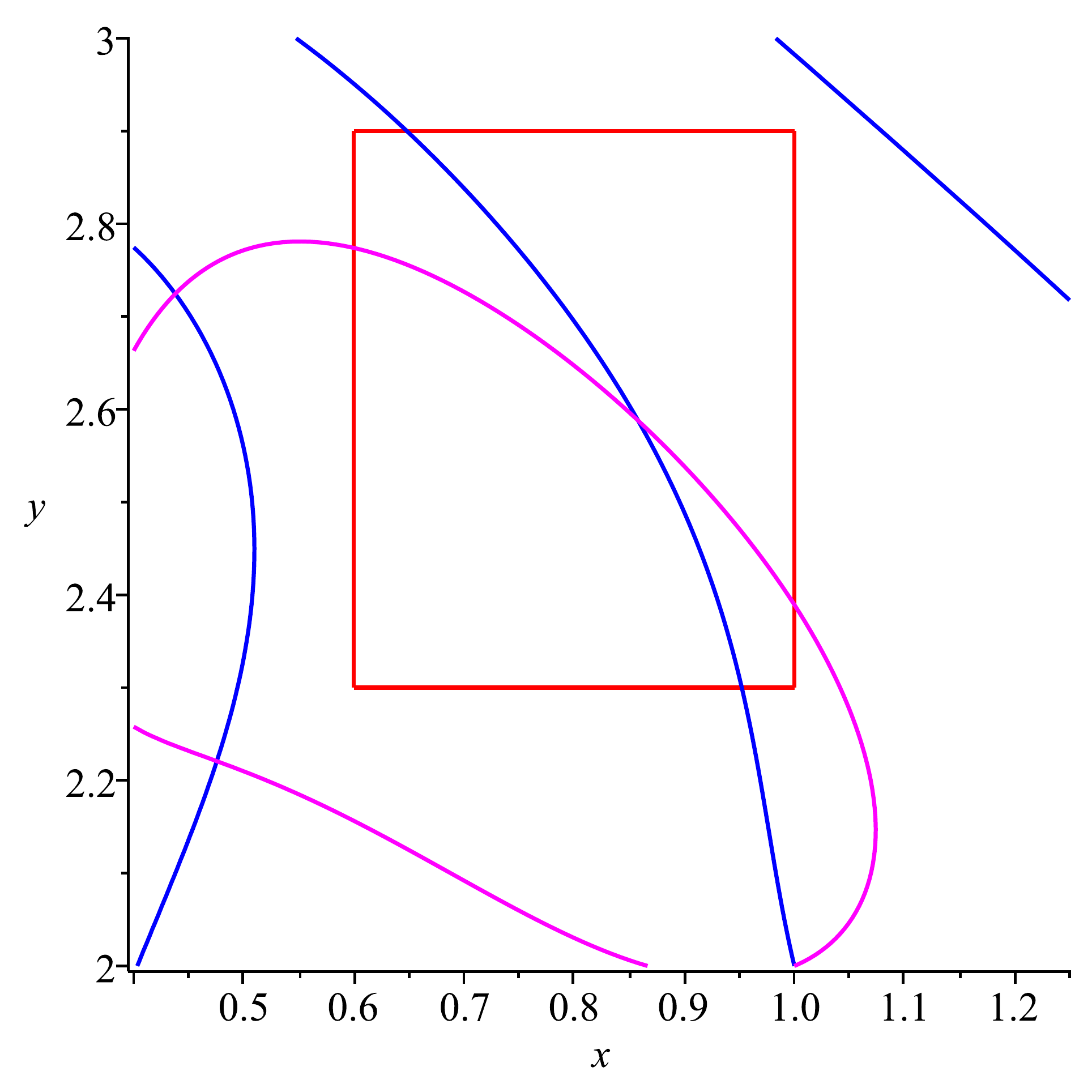}}
\caption{The PM box of a solution of system \eqref{E-TM-OP5} used in
the proof of Theorem \ref{T-TM} (in red). It corresponds to the
intersection of the curves defined by the curves $T_{5,1}(x,y)=0$
(in blue) and $T_{5,2}(x,y)=0$ (in magenta).} \label{F-FigPM-TM}
\end{figure}

We take, $g_1(y):=T_{5,1}(0.6,y)\cdot T_{5,1}(1,y)$ where
{\footnotesize
\begin{align*}
&T_{5,1}\left(0.6,y\right)= {\frac {16679880978201
{y}^{11}}{95367431640625}}-{\frac { 1167478199895987
{y}^{10}}{476837158203125}}+{\frac {6483464101392579
 {y}^{9}}{476837158203125}}\\
&-{\frac {86310474195914829 {y}^{8}}{ 2384185791015625}}+{\frac
{447004829071396386 {y}^{7}}{ 11920928955078125}}+{\frac
{8389294048378453266 {y}^{6}}{
298023223876953125}}\\
&-{\frac {148429346693234077422 {y}^{5}}{
1490116119384765625}}+{\frac {80721993185246247282 {y}^{4}}{
1490116119384765625}}+{\frac {317995739735111299953 {y}^{3}}{
7450580596923828125}}\\
&-{\frac {1653153957629818831467 {y}^{2}}{
37252902984619140625}}+{\frac {4929661459397407475239 y}{
931322574615478515625}}+{\frac{3868538667523044137292}{
4656612873077392578125}}.
\end{align*}}
\!\!and {\footnotesize
\begin{align*}
T_{5,1}\left(1,y\right)=& \left( y-2 \right)\times \\ &\left(
{y}^{10}-15 {y}^{9}+97 {y}^{8}-353 {y}^ {7}+792 {y}^{6}-1130
{y}^{5}+1022 {y}^{4}-566 {y}^{3}+177 {y}^{2} -27 y+1 \right) .
\end{align*}}
\!\!and prove that $g_1$  it is negative for $y\in[2.3,2.9].$ This can be
done by using the Sturm sequences of both polynomials.

Proceeding in an analogous way we obtain that
$g_2(y):=T_{5,2}(x,2.3)\cdot T_{5,2}(x,2.9)$ is a polynomial of
degree $62$ and it is negative for $x\in[0.6,1].$ Hence  the map
$f=\left(T_{5,1},T_{5,2}\right)$ satisfies the hypothesis of the PMT
and there exists a solution of system \eqref{E-TM-OP5} in
$\mathcal{B}$.

\medskip

(b) By imposing  $T^6(x,y)=(x,y)$, we get
\begin{equation}\label{E-TM-OP6}
x\cdot T_{6,1}(x,y)=0,\quad y\cdot T_{6,2}(x,y)=0,
\end{equation}
where $T_{6,1}$ and $T_{6,2}$ are polynomials with degree 63, and
967 and 910  monomials respectively. We compute the resultants of
these polynomials, and remove the repeated factors and those factors
corresponding to $x=0$ and $y=0$.

{\footnotesize
\begin{align*}
&P(x):=\frac{\mathrm{Res}(T_{6,1},T_{6,2};y)}{{x}^{420}\left(
x^2-3x+1  \right)\left( x-1 \right)^{2}\left( x-2 \right)^{405}}=
A\, \left( {x}^{3}-10 {x}^{2}+17 x-1 \right)  \left( {x}^{2}-3 x+1
\right)\\
&  \left( {x}^{6}-25 {x}^{5}+184 {x}^{4}-547 {x}^{3} +669
{x}^{2}-254 x+1 \right)  \left( 128 {x}^{12}-2816 {x}^{11}+
25280 {x}^{10}-124256 {x}^{9}\right.\\
&\left.+372768 {x}^{8}-713136 {x}^{7}+876616
 {x}^{6}-677024 {x}^{5}+309828 {x}^{4}-74692 {x}^{3}+7552 {x}^{2}
-272 x+1 \right)\\
&  \left( {x}^{12}-40 {x}^{11}+638 {x}^{10}-5436 {x }^{9}+27664
{x}^{8}-88424 {x}^{7}+181016 {x}^{6}-237152 {x}^{5}+
195072 {x}^{4}\right.\\
&\left. -96608 {x}^{3} +26624 {x}^{2}-3456 x+128 \right)
 \left( {x}^{6}-11 {x}^{5}+44 {x}^{4}-78 {x}^{3}+60 {x}^{2}-16 x+
1 \right) \left( x-1 \right)\\
&  \left( x -2 \right) \left( {x}^{3}-9 {x}^{2} +14 x-1 \right) ,
\end{align*} }
\!and {\footnotesize
\begin{align*}
&Q(y):=\frac{\mathrm{Res}(T_{6,1},T_{6,2};x)}{y^6(y-1)(y^2-3y+1)}=B\,\left(
y-2 \right)  \left( {y}^{3}+26 {y}^{2}+104 y+104 \right)
 \left( {y}^{3}+36 {y}^{2}+180 y+216 \right)\\
 &  \left( {y}^{6}+182 {y
}^{5}+3136 {y}^{4}+16072 {y}^{3}+25872 {y}^{2}+15680 y+3136
 \right) \left( 128 {y}^{12}+23808 {y}^{11}+602304 {y}^{10}+\right.\\
 & \left.
2820832 {y}^{9}-4126176 {y}^{8}-29841552 {y}^{7}+8077160 {y}^{6}+
52324032 {y}^{5}-24120108 {y}^{4}-9219772 {y}^{3}\right.\\
&\left.+3690240 {y}^{2}- 133920 y+837 \right)  \left( 16384
{y}^{12}+671744 {y}^{11}+5943296
 {y}^{10}+1502208 {y}^{9}-62922752 {y}^{8}\right.\\
 &\left.-53763840 {y}^{7}+
165704768 {y}^{6}+167848384 {y}^{5}-52858224 {y}^{4}-48703232 {y}^
{3}\right.\\
&\left.+5309928 {y}^{2}-133920 y+837 \right)  \left( y-1 \right)
 \left( {y}^{2}-3 y+1 \right),
\end{align*}
}\!where $A$ and $B$ are non-zero constants.

By using the Sturm method we obtain that $P(x)$ has 46 different
real roots (all of them positive), and $Q(x)$  has 40 different real
roots (16 of them positive). Hence, each  solution in the positive
quadrant of system \eqref{E-TM-OP6} is  contained in isolation in
one of the $736=46\times 16$ sets of the form
$$
\mathcal{I}_{i,j}=I_i\times J_j,\, i=1,\cdots,46;\, j=1,\ldots,16.
$$
where $\{I_i\subset\mathbb{R}^+,$ $i=1,\ldots, 46\}$ and
$\{J_j\subset\mathbb{R}^+,$ $j=1,\ldots, 16\}$  are intervals  with
rational ends such that each one of them contains a positive root of
$P$ and $Q$, respectively, in isolation.

In our computations we have obtained these intervals, with rational
ends and maximum length bounded by $10^{-100}$ (in order to apply
the discard procedure efficiently).  We don't give these intervals
in this paper. But in order to facilitate the reproduction of our
results and allow the reader to determine and locate the
$6$-periodic orbits, we indicate that \emph{these intervals (and
therefore the roots) are ordered, in the sense that if $\ell<m$ then
$I_\ell$ (respectively $J_\ell$) is completely to the left of  $I_m$
(respectively $J_m$). }

To discard those sets $\mathcal{I}_{i,j}$ that do not contain any
solution of system \eqref{E-TM-OP6}, we apply the discard method.
The procedure allows to eliminate $717$ boxes. Moreover, the box
$\mathcal{I}_{17,12}=[1,1]\times[2,2]$  corresponds with the fixed
point $(1,2)$ and the boxes $ I_{11,10},$ $I_{17,7},$ $I_{31,7},$
$I_{31,10},$ $I_{17,15}$ and $I_{11,15}$ correspond to the explicit
6-periodic orbit given in the statement. Hence the remaining
solutions of system \eqref{E-TM-OP6} must be contained in one of the
following $12$ non-discarded boxes:
\begin{equation}\label{E-caixesTM6}
\mathcal{I}_{8,6},\, \mathcal{I}_{9,3},\,   \mathcal{I}_{12,14},\,
\mathcal{I}_{13,11},\, \mathcal{I}_{16,16},\, \mathcal{I}_{18,1},\,
\mathcal{I}_{20,13},\,  \mathcal{I}_{21,8},\, \mathcal{I}_{30,4},\,
\mathcal{I}_{34,9},\,   \mathcal{I}_{35,2},\, \mathcal{I}_{37,5}.
\end{equation}
 Again, the PMT can be used to prove that in each of them
there is a solution of system \eqref{E-TM-OP6}. Since the solution
must be unique, we prove that in total there are $18$ periodic
points of minimal period $6$. Since the computation are quite
similar to the ones used to study the 5-periodic points we skip
them.~\end{proof}

\subsection{Determination of the $5$-periodic orbits}

By using the boxes computed in the proof of the above result, it is
easy to determine which points correspond to each orbit. Indeed,
first we concentrate on  the $5$-periodic orbits. Let us denote
$P_{i,j}=(x,y)$ the (unique) $5$-periodic point lying in  the box
$\mathcal{I}_{i,j}$  of \eqref{E-caixesTM}. We notice that the two
$5$-periodic orbits of $T$ in $\mathrm{Int}(\triangle)$ are given by
$ P_{9,10}\rightarrow P_{7,9}\rightarrow P_{8,6}\rightarrow
P_{14,3}\rightarrow P_{23,5},$ and $ P_{5,7}\rightarrow
P_{10,2}\rightarrow P_{20,1}\rightarrow P_{24,4}\rightarrow P_{6,1},
$ where
$$
\begin{array}{|c|}
\hline
\mbox{Orbit 1}\\
\hline
P_{9,10}\simeq\left(0.8581419568, 2.587834436 \right)\\
\hline
P_{7,9}\simeq\left(0.4754309022, 2.220729307 \right)\\
\hline
P_{8,6}\simeq\left(0.6198857282, 1.055803338\right)\\
\hline
P_{14,3}\simeq\left(1.440807176, 0.6544774210 \right)\\
\hline
P_{23,5}\simeq\left(2.744327621, 0.9429757645\right)\\
\hline
\end{array}
\quad
\begin{array}{|c|}
\hline
\mbox{Orbit 2}\\
\hline
P_{5,7}\simeq\left(0.3667104103, 1.192698099 \right)\\
\hline
P_{10,2}\simeq\left(0.8949903070, 0.4373748091 \right)\\
\hline
P_{20,1}\simeq\left(2.387507364, 0.3914462147 \right)\\
\hline
P_{24,4}\simeq\left(2.915257323, 0.9345807202 \right)\\
\hline
P_{6,11}\simeq\left(0.4377607446, 2.724543288\right)\\
\hline
\end{array}
$$
\begin{figure}[h]
\centerline{\includegraphics[scale=0.40]{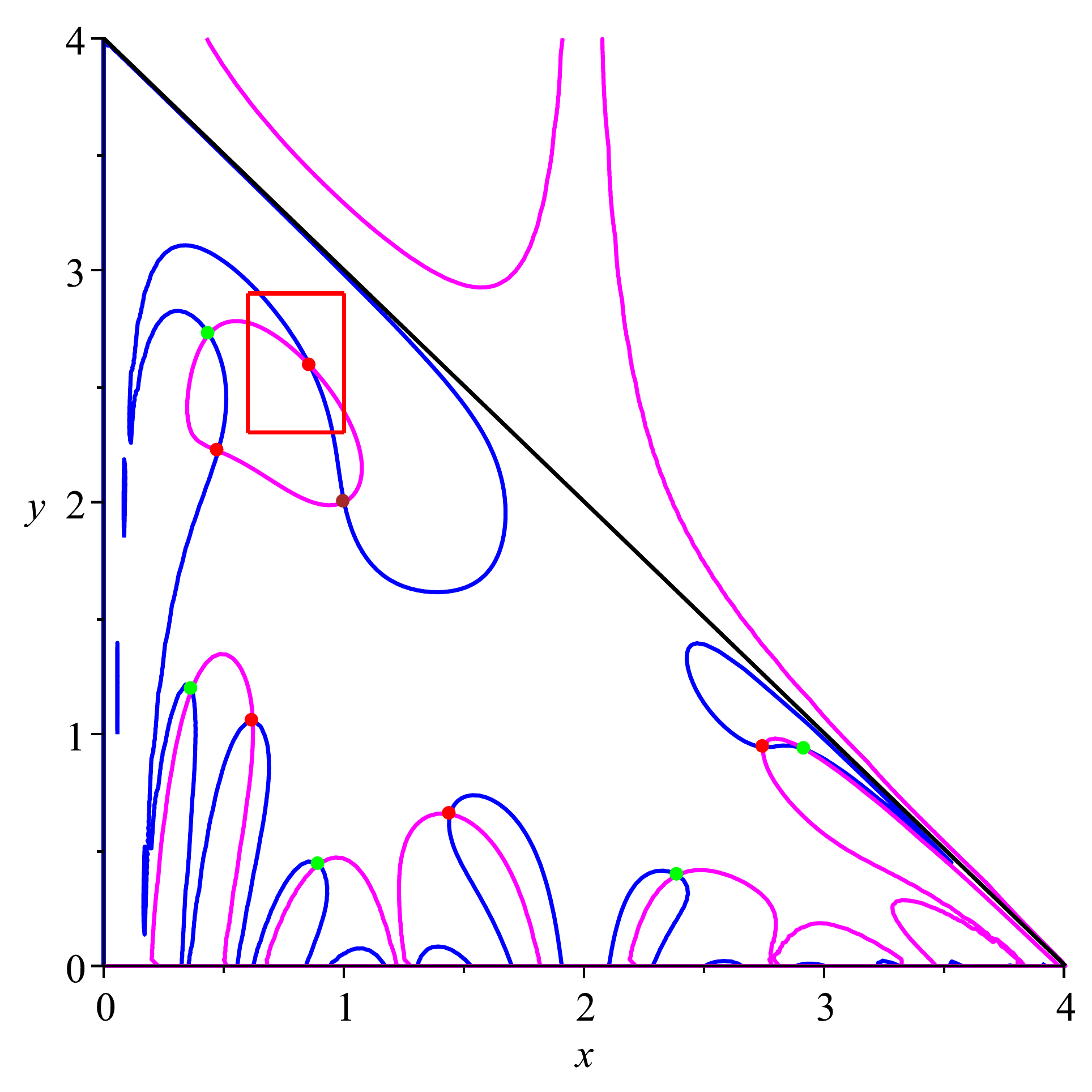}}
\caption{Two $5$-periodic orbits of the map \eqref{E-TMmap} in
$\mathrm{Int}(\triangle)$ (Orbit~1 in red and Orbit~2 in green).
They correspond to the intersection of the curves defined by the
curves $T_{5,1}(x,y)=0$ (in blue) and $T_{5,2}(x,y)=0$ (in magenta).
The fixed point $(1,2)$ (in brown). The PM box containing the point
$P_{9,10}$ used in the proof of Theorem~\ref{T-TM} (in red).}
\label{F-FigTM}
\end{figure}

These decimal approximations have obtained using the intervals given
in the appendix. Remember that they give an approximation with a
maximum error of $10^{-40}.$ The points are depicted in Figure
\ref{F-FigTM}.

The above assertions can be proved, by using the fact that taking
$\mathcal{I}_{i,j}=[\underline{x},\overline{x}]\times
[\underline{y},\overline{y}]$, and setting $(\widetilde x,\widetilde
y)=T(P_{i,j})$, since $\widetilde y=xy$ it must satisfy
$\underline{x}\,\underline{y}<\widetilde
y<\overline{x}\,\overline{y}$, and from these inequalities is easy
to identify in which box of \eqref{E-caixesTM} is $(\widetilde
x,\widetilde y)$. For instance, for the point $P_{9,10}\in
\mathcal{I}_{9,10}$, and setting $\widetilde y=T(P_{9,10})_2$, one
has that $a:=\underline{x}\,\underline{y}<\widetilde
y<b:=\overline{x}\,\overline{y}$ where {\scriptsize
$$
a= \frac{
16852116181629561083016478047392649752537328677768476306665633800018763276864769965
}{
7588550360256754183279148073529370729071901715047420004889892225542594864082845696
},$$} {\scriptsize
$$
b= \frac{
4213029045407390270754119511848162438134407216162333392848558471828944039611220943
}{
1897137590064188545819787018382342682267975428761855001222473056385648716020711424
}.$$} Now, it is easy to check that the only interval $J_j$ for
$j\in\{1,2,\ldots,10,11\}$ with nonempty intersection with $(a,b)$
is $J_9$, hence $P_{7,9}=T(P_{9,10})$.

\subsection{Determination of the $6$-periodic orbits}

Proceeding as in the previous section, we determine the points of
two of  the $6$-periodic orbits. The third one is explicit. Again we
denote $P_{i,j}=(x,y)$ the  (unique) $6$-periodic point lying in the
box $\mathcal{I}_{i,j}$  of \eqref{E-caixesTM6}. The points are
depicted in Figure \ref{F-FigTM6}.

\begin{figure}[h]
\centerline{\includegraphics[scale=0.40]{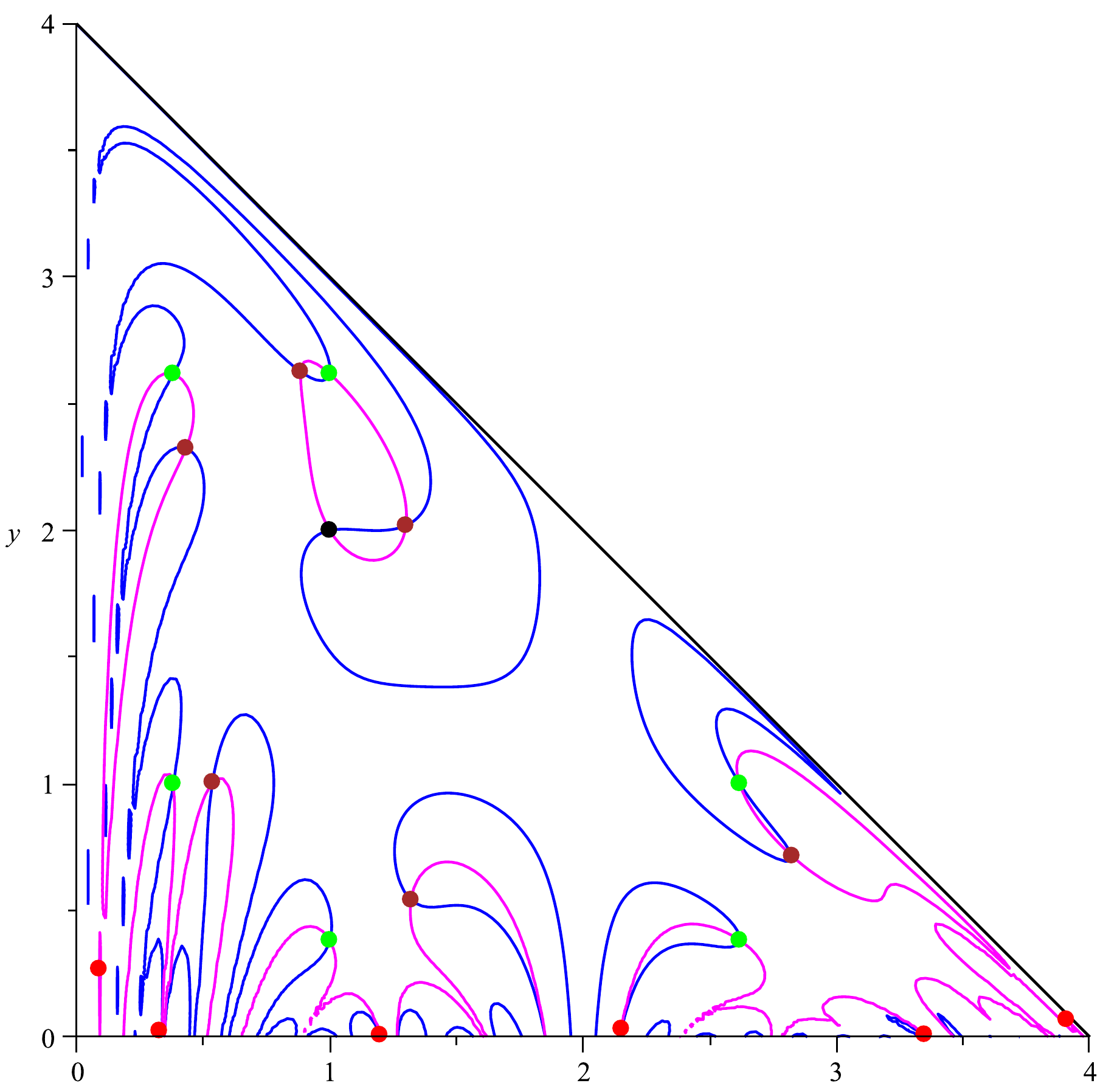}}
\caption{Three $6$-periodic orbits of the map \eqref{E-TMmap}  in
$\mathrm{Int}(\triangle)$ (Orbit 1,2 and 3 in brown, red and green,
respectively). They correspond to the intersection of the curves
defined by the curves $T_{6,1}(x,y)=0$ (in blue) and
$T_{6,2}(x,y)=0$ (in magenta). The fixed point $(1,2)$ (in black).}
\label{F-FigTM6}
\end{figure}

We have, Orbit 1: $ P_{20,13}\rightarrow P_{16,16}\rightarrow
P_{12,14}\rightarrow P_{13,11}\rightarrow P_{21,8}\rightarrow
P_{34,9};$ Orbit 2: $ P_{8,6}\rightarrow P_{9,3}\rightarrow
P_{18,1}\rightarrow P_{35,2}\rightarrow P_{30,4}\rightarrow
P_{37,5}; $ and Orbit 3: $ P_{11,10}\rightarrow P_{17,7}\rightarrow
P_{31,7}\rightarrow P_{31,10}\rightarrow P_{17,15}\rightarrow
P_{11,15}, $ where the points of Orbit 3 are the ones of the
statement and
$$
\begin{array}{|c|}
\hline
\mbox{Orbit 1}\\
\hline
P_{20,13}\simeq\left(1.300802119, 2.018868702 \right)\\
\hline
P_{16,16}\simeq\left(0.8849736378, 2.626148686 \right)\\
\hline
P_{12,14}\simeq\left(0.4326438557, 2.324072356\right)\\
\hline
P_{13,11}\simeq\left(0.5378990919, 1.005495625 \right)\\
\hline
P_{21,8}\simeq\left(1.321405751, 0.5408551836\right)\\
\hline
P_{34,9}\simeq\left(2.824820695, 0.7146891501\right)\\
\hline
\end{array}
\quad
\begin{array}{|c|}
\hline
\mbox{Orbit 2}\\
\hline
P_{8,6}\simeq\left(0.09022635321, 0.2685661106 \right)\\
\hline
P_{9,3}\simeq\left(0.3285328773, 0.02423174076 \right)\\
\hline
P_{18,1}\simeq\left(1.198236734, 0.007960923513\right)\\
\hline
P_{35,2}\simeq\left(3.347636595, 0.009539070990 \right)\\
\hline
P_{30,4}\simeq\left(2.151942266, 0.03193334313\right)\\
\hline
P_{37,5}\simeq\left(3.908194837, 0.06871871076\right)\\
\hline
\end{array}
$$
The  decimal approximations have obtained using the intervals
computed  in the proof of Theorem \ref{T-TM}. They gave an
approximation with a maximum error of $10^{-100}.$ The points are
depicted in Figure \ref{F-FigTM6}.

\section{Limit cycles of piecewise linear differential systems}\label{S-Piecewise}

The study of the number  of limit cycles for planar differential
systems is a classical topic in the theory of dynamical systems. In
the last years, many attention has been devoted to the study of
nested limit cycles of piecewise linear systems, steered by the
applicability of these systems in the modelling of biological and
mechanical applications. In 2012, S.M.~Huan and X.S.~Yang gave
numerical evidences of a piecewise linear system with two zones and
a discontinuity straight line, having three nested limit cycles
(\cite{HY}). A  proof based on the Newton--Kantorovich theorem of
the existence of these limit cycles for this example and a nearby
one, was given by J.~Llibre and E.~Ponce (\cite{LP}). A different
proof, from a bifurcation viewpoint, was presented by E.~Freire,
E.~Ponce and F.~Torres in \cite{FPT12}. Until now, as far as we
know, three is the maximum observed number of limit cycles in
piecewise linear differential systems with two zones and a
discontinuity straight line, but it is not known if this is the
maximum number that such type of systems can have.

 In this section we present a new example, again with 3 limit cycles,
inspired on the ones given in \cite{HY,LP}. The main difference is
that our proof of their existence is based on the PMT.

\begin{teo}\label{T-3LC}
The two-zones piecewise linear differential system
\begin{equation}\label{E-PLS}
\dot{\mathbf{x}}=\left\{\begin{array}{ll}
A^+\mathbf{x} & \mbox{if } x\geq 1,\\
A^-\mathbf{x} & \mbox{if } x\leq 1,
\end{array}\right.
\end{equation} where $\mathbf{x}=(x,y)^t$,
$$
A^-:=\left(\begin{array}{cc}
\frac{67}{50} & -\frac{833}{125}\\
\frac{1}{2} & -\frac{87}{50}
\end{array}\right)
\,\mbox{ and }\, A^+:=\left(\begin{array}{cr}
\frac{3}{8} & -1\\
1 & \frac{3}{8}
\end{array}\right),
$$
has at least three nested hyperbolic limit cycles surrounding the
origin.
\end{teo}

To prove the above result, we will use systematically the following
lemma, that is a straightforward consequence of Taylor's formula.

\begin{lem}\label{L:acotacio}
Set  $h(x)=A\cos(\alpha x)+B\sin(\alpha x)+C\e^{\beta x}+D\e^{-\beta
x}$,  with $A,B,C,D\in\R$, $\alpha\neq 0,$  $\beta>0$ and
 $x\in[\underline{x},\overline{x}]\subset\R^+$. Then for each
$n\geq 0$ we have  $h(x)=\sum\limits_{j=0}^n a_j x^j+m_n(x)
x^{n+1}$, where
\begin{align}
&a_j=\frac{1}{j!}\left(\alpha^j\left[A\cos\left(j\frac{\pi}{2}\right)+B\sin\left(j\frac{\pi}{2}\right)\right]
+\beta^j\left[C+(-1)^jD\right]\right),\nonumber\\
&\left|m_n(x)\right|\leq
\overline{m}=\frac{\left|\alpha\right|^{n+1}\left(|A|+|B|\right)+
\left|\beta\right|^{n+1}\left(|C|\e^{\beta\overline{x}}+|D|\e^{-\beta\underline{x}}\right)}{(n+1)!}\label{lem2}.
\end{align}
\end{lem}

\begin{proof}[Proof of Theorem \ref{T-3LC}] Let $\varphi^\pm(t;p)=(x^\pm(t;p),x^\pm(t;p))$
denote the flows associated to the linear systems
$\dot{\mathbf{x}}=A^\pm \mathbf{x}$. Observe
 that if there exists a limit cycle then it must lie on both sides of the line $x=1$, so
  let $t^->0$ be the smaller time  such that $x^-(t^-;(1,y))=1$ for a point $(1,y)$ with $y>0$,
   and let $t^+>0$ be the smaller time such that $x^+(-t^+;(1,y))=1$. Then any limit cycle must satisfy $x^+(-t^+;(1,y))-1=0,$
$x^-(t^-;(1,y))-1=0,$ $y^+(-t^+;(1,y))-y^-(t^-;(1,y))=0,$ or
equivalently
\begin{align}
&\e^{-\frac{3}{8} u} \left( \cos \left( u \right) +y\sin \left( u
 \right)  \right) -1=0,\label{E-eq1}\\
&\frac{1}{35}  \left( 35 \cos \left( {\frac {49}{50}}  v\right)+
\left( -238 y+55 \right) \sin \left( {\frac {49}{50}}  v\right)
\right)
\e^{-\frac{1}{5}v}-1=0,\label{E-eq2}\\
&\frac{1}{49}  \left(-49 \cos \left( {\frac {49}{50}  v} \right) y+
\left( 77 y-25 \right) \sin \left( {\frac {49}{50} v}
 \right)  \right) {
{\rm e}^{-\frac{v}{5}}}+{{\rm e}^{-\frac{3}{8} u}} \left( \cos
\left( u \right) y- \sin \left( u \right)  \right) =0,\label{E-eq3}
\end{align}
where $u=t^+>0$ and $v=t^->0$.

By solving equation \eqref{E-eq1} we get $
y=\big(\e^{-\frac{3}{8}u}-\cos(u)\big)/{\sin(u)}. $
 By substituting this expression in equations \eqref{E-eq2}  and \eqref{E-eq3}, we obtain
\begin{equation}\label{E-g1g2}
\begin{array}{ll}
g_1(u,v):=&a(v)\cos(u)+b(v)\sin(u)-a(v)\e^{\frac{3}{8}u}=0,\\
g_2(u,v):=&c(v)\cos(u)+d(v)\sin(u)+e(v)\e^{\frac{3}{8}u}+f(v)\e^{-\frac{3}{8}u}=0,
\end{array}
\end{equation}
where
$$
\begin{array}{ll}
a(v)=238\,{{\rm e}^{-\frac{v}{5}}}\sin \left( {\frac {49}{50}\,v}
\right) ,& b(v)=55\,{{\rm e}^{-\frac{v}{5}}}\sin \left( {\frac
{49}{50}\,v} \right) +35\,{ {\rm e}^{-\frac{v}{5}}}\cos \left(
{\frac {49}{50}\,v} \right) -35
,\\
\end{array}
$$
and
$$\begin{array}{ll}
c(v)=49\,{{\rm e}^{-\frac{v}{5}}}\cos \left( {\frac {49}{50}\,v}
\right)-77\,{{\rm e}^{-\frac{v}{5}}}\sin \left( {\frac {49}{50}\,v}
\right)  +49,&
d(v)=-25\,{{\rm e}^{-\frac{v}{5}}}\sin \left( {\frac {49}{50}\,v} \right), \\
e(v)=77\,{{\rm e}^{-\frac{v}{5}}}\sin \left( {\frac {49}{50}\,v}
\right) -49\, {{\rm e}^{-\frac{v}{5}}}\cos \left( {\frac
{49}{50}\,v} \right), & f(v)=-49.
\end{array}
$$

\begin{figure}[h]
\centerline{\includegraphics[scale=0.26]{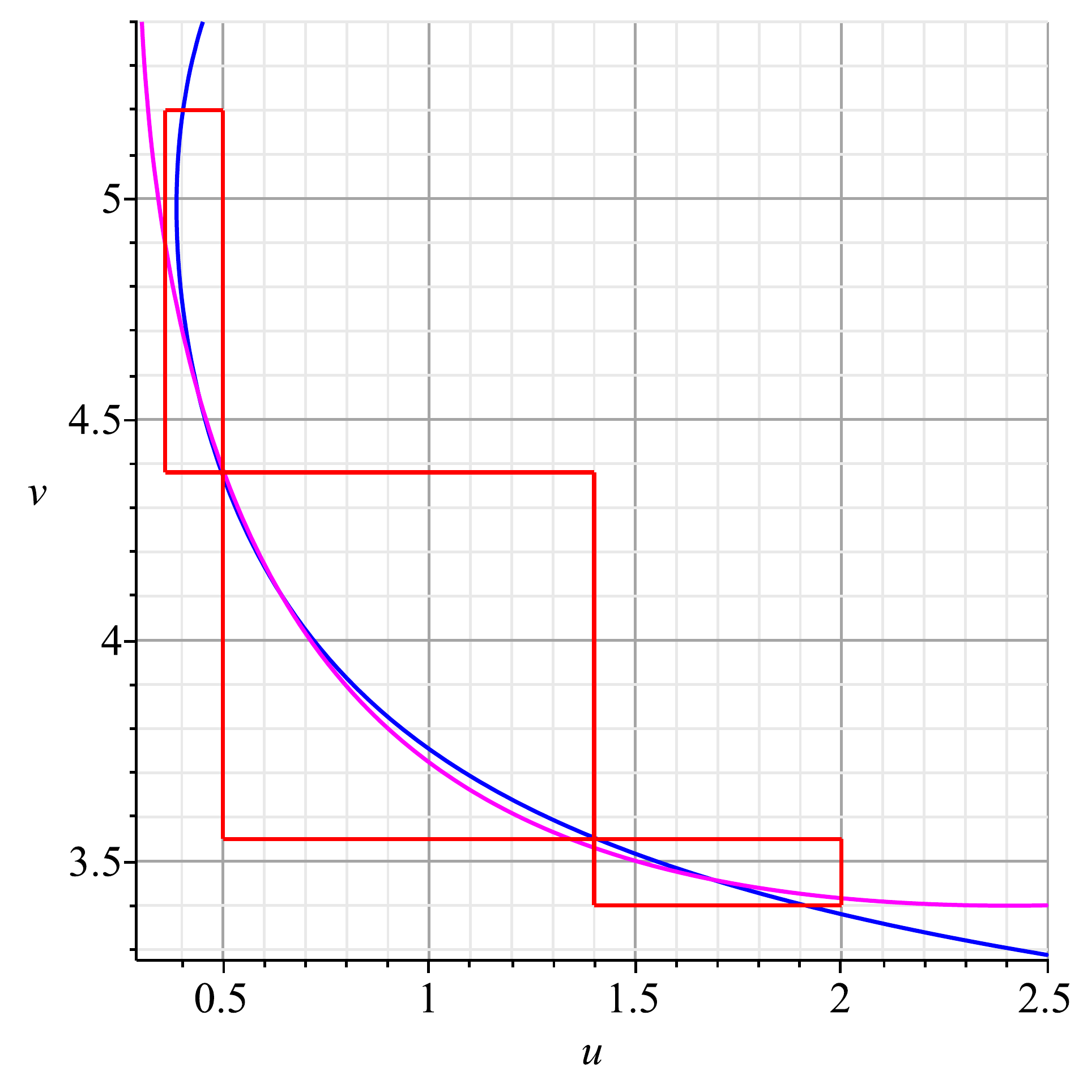}\qquad\includegraphics[scale=0.23]{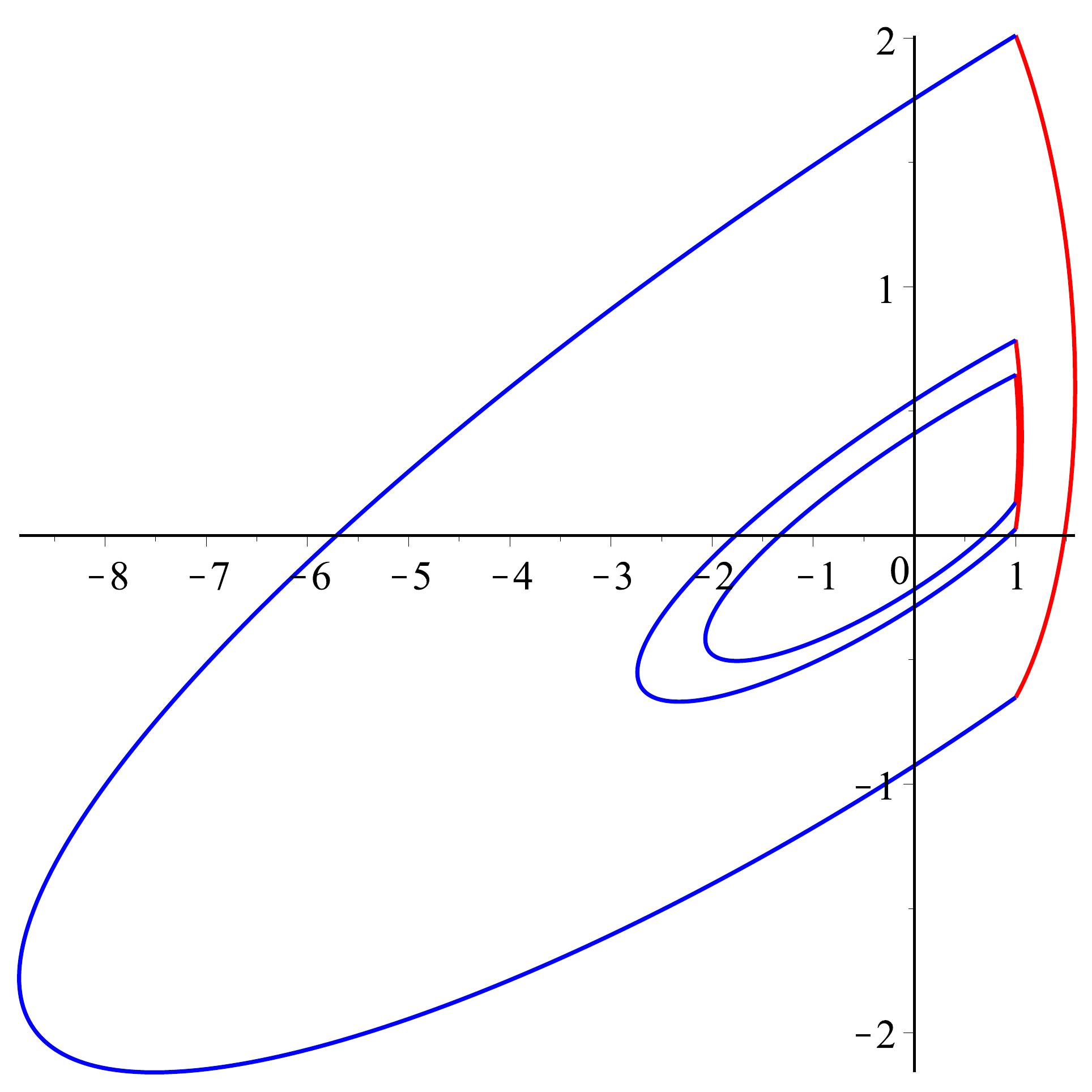}}
\caption{Left part: Intersection points between   $g_{1}(u,v)=0$ (in
blue) and $g_{2}(u,v)=0$ (in magenta) and some PM boxes containing
them. Right part: the 3 limit cycles of system \eqref{E-PLS}.}
\label{F-3LC}
\end{figure}

Numerically it is easy to guess that there are 3 different solutions
of system~\eqref{E-g1g2}, see Figure \ref{F-3LC}. Their approximate
values in $(u,v)$ variables are $(0.441441, 4.554696),$
$(0.639391,4.105752)$ and $(1.686596, 3.458345).$ Once we prove that
near  these values there are actual solutions of
system~\eqref{E-g1g2}, each one of them will correspond to a
solution of the system of equations \eqref{E-eq1}--\eqref{E-eq3}
and, consequently,  all them will give rise to  3 limit cycles of
system \eqref{E-PLS}, see again Figure \ref{F-3LC}.

To prove the existence of 3 solutions  of system~\eqref{E-g1g2}, we
 consider the 3 boxes:
$$
\mathcal{B}_1:=\left[{\frac{9}{25}},\frac{1}{2}\right]\times\left[{\frac{219}{50}},{\frac{26}{5}}\right],\,
\mathcal{B}_2:=\left[\frac{1}{2},\frac{7}{5}\right]\times\left[\frac{71}{20},{\frac{219}{50}}\right]
\,
\mathcal{B}_3:=\left[\frac{7}{5},2\right]\times\left[{\frac{17}{5}},{\frac{71}{20}}\right]
$$
and prove that they are PM boxes for $(g_1,g_2).$

 To see that we are
under the hypotheses of the PMT, in all the cases we proceed
systematically in the following form:
 We suppose that we want to prove that a function $h(x)$ of the form of Lemma \ref{L:acotacio} is positive
 (resp. negative)  in $[\underline{x},\overline{x}]\subset \R^+$.  Firstly, we use this lemma  to have its Taylor
  polynomial  at $x=0$ up to a certain order $n.$ Secondly,  we minorize (resp. majorize) the polynomial
   by a \emph{polynomial with rational coefficients}, obtained by truncating the decimal
    expression up to some suitable order $k$, and subtracting (resp. adding) $10^{-k}$ to the
    obtained quantity, that is
 $$
 a_j^\pm:=\mathrm{Trunc}(a_j\cdot 10^k)\cdot 10^{-k}\pm 10^{-k}\in\mathbb{Q},
 $$
where $\mathrm{Trunc}$ stands for the truncation to the next nearest
integer towards
 $0$. Finally  we consider $
 P_{n,k}^{\pm}(x)=\sum\limits_{j=0}^n  a_j^{\pm} x^j \pm M x^{n+1}$, where $M\in\mathbb{Q}$
 is a suitable upper bound of the right-hand side expression  in \eqref{lem2}, so that
 $$
 P_{n,k}^{-}(x)-Mx^{n+1}\leq h(x) \leq P_{n,k}^{+}(x)+Mx^{n+1}.
 $$
 Now we only have to check if $P_{n,k}^{-}(x)>0$ (resp. $P_{n,k}^{+}(x)<0$) in
  $[\underline{x},\overline{x}]$. To do this we use the Sturm sequences of  these polynomials.

 Applying this approach we prove that  $\mathcal{B}_1$,
$\mathcal{B}_2$ and $\mathcal{B}_3$ are PM boxes by setting the
following parameters $n$, $k$ and $M$ in each face (we use the
notation $\mathcal{B}=[\underline{u},\overline{u}]\times
[\underline{v},\overline{v}]$):

\noindent Box $\mathcal{B}_1$:
$$
\begin{array}{|r|c|c|c|c|}
\hline
\mathrm{Face} & u=\underline{u} & u=\overline{u} & v=\underline{v} & v=\overline{v}\\
\hline
\mbox{Target function } h \mbox{ and sign} & g_1<0&g_1>0&g_2>0&g_2<0\\
\hline
\mbox{Polynomial }  & P_{n,k}^+& P_{n,k}^-& P_{n,k}^-& P_{n,k}^+\\
\hline
\mbox{Parameters: } n & 16& 16& 4 & 4 \\
 k & 15& 15& 3&  3\\
 \hline
\mbox{Bound } M & 10^{-13}& 10^{-12}& \frac{7}{10}& \frac{4}{5}\\
 \hline
\end{array}
$$

\noindent Box $\mathcal{B}_2$:
$$
\begin{array}{|r|c|c|c|c|}
\hline
\mathrm{Face} & u=\underline{u} & u=\overline{u} & v=\underline{v} & v=\overline{v}\\
\hline \mbox{Target function } h \mbox{ and sign} &
\e^{\frac{v}{5}}\cdot g_2>0&
\e^{\frac{v}{5}} \cdot g_2<0&g_1<0&g_1>0\\
\hline
\mbox{Polynomial }  & P_{n,k}^-& P_{n,k}^+& P_{n,k}^+& P_{n,k}^-\\
\hline
\mbox{Parameters: } n & 16& 16& 6 & 4 \\
 k & 14& 10& 2&  2\\
  \hline
\mbox{Bound } M & 10^{-13}& 10^{-12}& \frac{1}{50}& \frac{13}{10}\\
 \hline
\end{array}
$$

  \noindent Box $\mathcal{B}_3$:
$$
\begin{array}{|r|c|c|c|c|}
\hline
\mathrm{Face} & u=\underline{u} & u=\overline{u} & v=\underline{v} & v=\overline{v}\\
\hline \mbox{Target function } h \mbox{ and sign} &
\e^{\frac{v}{5}}\cdot g_1<0&
\e^{\frac{v}{5}} \cdot g_1>0&g_2>0&g_2<0\\
\hline
\mbox{Polynomial }  & P_{n,k}^+& P_{n,k}^-& P_{n,k}^-& P_{n,k}^+\\
\hline
\mbox{Parameters: } n & 13& 11& 6 & 7 \\
 k & 8& 8& 3&  2\\
  \hline
\mbox{Bound } M & 10^{-8}& 10^{-6}& 10^{-2}& 10^{-2}\\
 \hline
\end{array}
$$

We only give details of some of the computations for
$\mathcal{B}_1$. For instance we show that $g_2(u,\underline{v})>0$
for all
 $u\in[\underline{u},\overline{u}]$.  All the other computations can be reproduced
 using the information given above. Indeed, we take $\mathcal{B}_1$ and we observe that
$$
g_2(u,\underline{v})=c\left({\frac{219}{50}}\right)\cos(u)+
d\left(\frac{219}{50}\right)\sin(u)
+e\left({\frac{219}{50}}\right)\e^{\frac{3}{8}u}+f\left({\frac{219}{50}}\right)\e^{-\frac{3}{8}u},
$$
which has the form of the function $h$ in Lemma \ref{L:acotacio},
where
$$\begin{array}{l}
 A=c\left(\frac{219}{50}\right)=49\,{{\rm e}^{-{\frac{219}{250}}}}\cos
 \left( {\frac{10731}{2500}}\right) -77\,{{\rm e}^{-{\frac{219}{250}}}}\sin \left( {\frac{10731}{2500}} \right) +49,\\
B=d\left(\frac{219}{50}\right)=-25\,{{\rm e}^{-{\frac{219}{250}}}}\sin \left( {\frac{10731}{2500}}\right),\\
C=e\left(\frac{219}{50}\right)=\left( -49\,\cos \left(
{\frac{10731}{2500}} \right) +77\,\sin
 \left( {\frac{10731}{2500}} \right)  \right) {{\rm e}^{-{\frac{219}{250}}}},\\
D=f\left(\frac{219}{50}\right)=-49.
 \end{array}
$$
By applying this lemma with $n=4,$ we obtain $P_4(u)=\sum_{j=0}^4
a_j u^j.$ After taking $a_j^-:=\mathrm{Trunc}(a_j\cdot 10^3)\cdot
10^{-3}- 10^{-3}$ for each $j=0,\ldots,4$ we get
$$
P_{4,3}^-(u)=-{\frac{1}{1000}}+{\frac {1001}{50}\,u}-{\frac
{39899}{1000}\,{u}^{2}} -{\frac {669}{500}\,{u}^{3}}+{\frac
{357}{125}\,{u}^{4}}.
$$
By using \eqref{lem2} we also obtain that
\begin{multline*}
\overline{m}_4= \left( {\frac {49\,{{\rm
e}^{-{\frac{219}{250}}}}}{120}}+{\frac {3969 \,{{\rm
e}^{-{\frac{1377}{2000}}}}}{1310720}} \right) \cos \left( {
\frac{10731}{2500}} \right)+ \\ \left( -{\frac {17\,{{\rm
e}^{-{\frac{ 219}{250}}}}}{20}}-{\frac {6237\,{{\rm
e}^{-{\frac{1377}{2000}}}}}{ 1310720}} \right) \sin \left(
{\frac{10731}{2500}} \right) +  {\frac{49} {120}}+{\frac
{3969\,{{\rm e}^{-{\frac{27}{200}}}}}{1310720}} \simeq
0.6664<\frac{7}{10}=M.
\end{multline*}

By using the Sturm sequence of $P_{4,3}^-(u)-Mu^5$ we
 prove that it has no roots  in
$[\underline{u},\overline{u}]$ and, moreover, it is positive in this
interval. Hence $ 0< P_{4,3}^-(u)-Mu^5< g_2(u,\underline{v})\,\mbox{
for all } u\in [\underline{u},\overline{u}]. $

To prove the hyperbolicity of the limit cycles we can follow the
same ideas that in \cite{LP}.~
\end{proof}

\section{On the existence of a symmetric central configuration}\label{S-central}

 Central configurations
  are a very special type of solutions of the $N$-body problem in Celestial Mechanics,
  in which the acceleration of every body is proportional to the position vector of the
   body with respect the center of mass of the system. They play an important role in
    practical applications and there is a vast literature on the topic, both classical and
     recent. An account of known facts and open problems can be found in
     \cite{L}.

In the  $(1+n)$-body problem it is supposed that there is one body
with a large  mass and $n$ bodies whose masses can be neglected in
comparison with the large one. These bodies are named as
infinitesimal masses. With our approach we prove in a very simple
way the
     existence of a special planar central configurations
 of the $(1+4)$-body problem, already given in \cite{CLO}.

According to the results in \cite{CLN94,CLO}, all planar central
configurations in the $(1+n)$-body problem lie on a circle centered
at the position of the large mass. Furthermore, denoting by
$\alpha_i\in \mathbb{S}^1,$ with
$\alpha_1<\alpha_2<\ldots<\alpha_n,$ the angles defined by the
position of the $i$th infinitesimal masses on a circle centered at
the origin, central configurations must satisfy the system of $n$
equations, $i=1,\ldots,n,$
\begin{equation*}
\sum\limits_{j=1}^{n} f(\alpha_j-\alpha_i)=0,\quad\mbox{with}\quad
f(\theta):=\Big(1-\dfrac{\sqrt{2}}{4\sqrt{(1-\cos(\theta))^3}}\Big)\sin(\theta).
\end{equation*}
Notice that $f(-\theta)=-f(\theta).$

When $n=4,$ we introduce the variables $u=\alpha_2-\alpha_1,$
$v=\alpha_3-\alpha_1$ and  $w=\alpha_4-\alpha_1.$ Then the above
system is equivalent to the system (with only 3 equations):
\[
f(u)+f(v)+f(w)=0, -f(u)+f(v-u)+f(w-u)=0, -f(v)+f(u-v)+f(w-v)=0.
\]

Although our point of view could be applied to prove the existence
of solutions of the above system (and so of central configurations),
for simplicity we will look for a symmetric one, the one satisfying
tat $\alpha_4-\alpha_3=\alpha_2-\alpha_1$. In our coordinates this
implies that $w=u+v$ and hence the system reduces to the system with
2 equations
\[
g_1(u,v):=f(u)+f(v)+f(u+v)=0,\quad g_2(u,v):=f(u)-f(v)-f(v-u)=0.
\]

\begin{figure}[h]
\centerline{\includegraphics[scale=0.35]{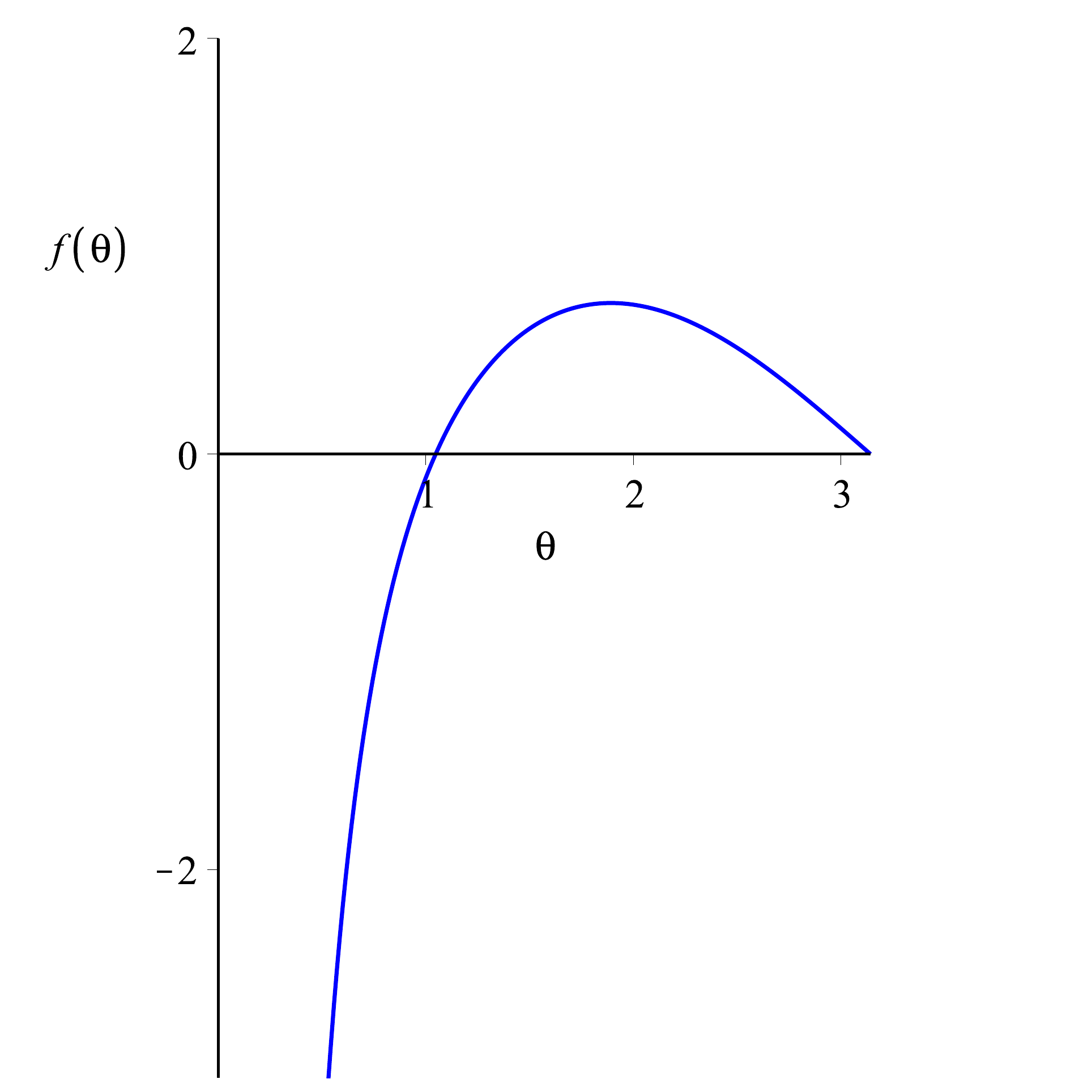}\!\!\!\includegraphics[scale=0.35]{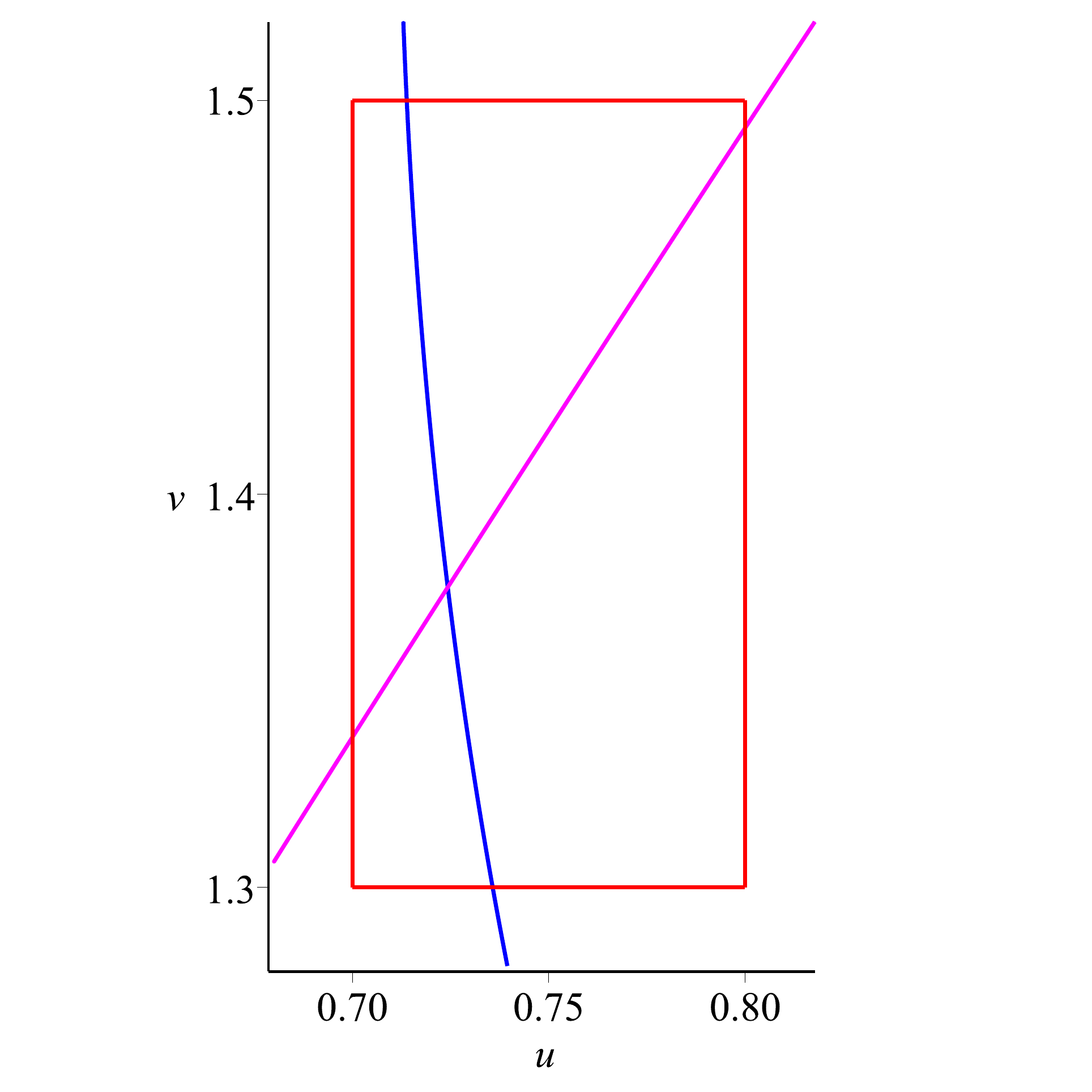}}
\caption{Left figure: graph of $f(\theta)$ in $(0,\pi]$. Right
figure: intersection of the curves $g_1(u,v)=0$ (in blue) and
$g_2(u,v)=0$ (in magenta). In red, a PM box.} \label{F-CC}
\end{figure}

Let us prove that we can apply PMT to the box
$[0.7,0.8]\times[1.3,1.5],$ see Figure \ref{F-CC}. To simplify the
notation we denote by $g_i(I)+g_j(J)$ the set of all values
$g_i(x)+g_j(y),$ with $x\in I$ and $y\in J.$  Then, simply using
that $f$ is increasing between in $(0,\theta^*)$ and decreasing in
$(\theta^*,\pi)$, where $\theta^*\simeq 1.891\in(1.8,1.9)$, see
again Figure \ref{F-CC}, we get:
\begin{align*}
\left.g_1(0.7,v)\right|_{v\in[1.3,1.5]}&
=\left.f(0.7)+f(v)+f(0.7+v)\right|_{v\in[1.3,1.5]}\\
&=
f(0.7)+f([1.3,1.5])+f([2,2.2])\\
&\leq f(0.7)+f(1.5)+f(2)\simeq -0.031<0.
\end{align*}
 Similarly, $\left.g_1(0.8,v)\right|_{v\in[1.3,1.5]}=
f(0.8)+f([1.3,1.5])+f([2.1,2.3])\geq f(0.8)+f(1.3)+f(2.3)\simeq
0.24>0,$ $\left.g_2(u,1.3)\right|_{u\in[0.7,0.8]}=
f([0.7,0.8])-f(1.3)-f([0.5,0.6])\geq f(0.7)-f(1.3)-f(0.6)\simeq
0.40>0,$ and
 $\left.g_2(u,1.5)\right|_{u\in[0.7,0.8]}=
f([0.7,0.8])-f(1.5)-f([0.7,0.8])\leq f(0.8)-f(1.5)-f(0.7)\simeq
-0.052<0.$ Hence we have proved the existence of a symmetric central
configuration for this problem. In fact, numerically this solution
is $u =u^*\simeq 0.7242718590,$ $v =v^*\simeq 1.376255451$ and it
corresponds with the values $x^*=\cos((u^*+v^*)/4)\simeq 0.86525786$
and $y^*=\sin((v^*-u^*)/4)\simeq 0.162275119$ given in the proof of
\cite[Prop 13]{CLO} and found using the  variables
$x=\cos((\alpha_2+\alpha_3-2\alpha_1)/4)$  and
$y=\sin((\alpha_3-\alpha_2)/4).$


\section*{Appendix.}

The 32 intervals of maximum length $10^{-40}$ containing in
isolation the positive roots of  the polynomial $P(x)$  that appear
in the proof of Theorem \ref{T-PM-n} are: {\tiny
\begin{align*}
&I_{1}=\left[{\frac{30416172743817963645774671148052190959399}{
178405961588244985132285746181186892047843328}},{\frac{
60832345487635927291549342296104381918799}{
356811923176489970264571492362373784095686656}}\right],\\[0.1cm] &
I_{2 }=,\left[{\frac{ 61004111234951360792324423318169046946281}{
356811923176489970264571492362373784095686656}},{\frac{
30502055617475680396162211659084523473141}{
178405961588244985132285746181186892047843328}}\right],\\[0.1cm] &
I_{ 3}=,\left[{\frac{ 27421400457977260353638941607050001933877}{
11150372599265311570767859136324180752990208}},{\frac{
54842800915954520707277883214100003867755}{
22300745198530623141535718272648361505980416}}\right],\\[0.1cm] &
I_{4 }=,\left[{\frac{ 55177596464914226273769532370430091204151}{
22300745198530623141535718272648361505980416}},{\frac{
6897199558114278284221191546303761400519}{
2787593149816327892691964784081045188247552}}\right],\\[0.1cm] &
I_{5 }=,\left[{\frac{ 63889964233534261969205802666073326201313}{
174224571863520493293247799005065324265472}},{\frac{
31944982116767130984602901333036663100657}{
87112285931760246646623899502532662132736}}\right],\\[0.1cm] &
I_{6 }=,\left[{\frac{ 76268678313703396955788049048149760899577}{
174224571863520493293247799005065324265472}},{\frac{
38134339156851698477894024524074880449789}{
87112285931760246646623899502532662132736}}\right],\\[0.1cm] &
I_{7 }=,\left[{\frac{ 82831745394319077306596661336176030777141}{
174224571863520493293247799005065324265472}},{\frac{
41415872697159538653298330668088015388571}{
87112285931760246646623899502532662132736}}\right],\\[0.1cm] &
I_{8 }=,\left[{\frac{ 26999831398533091510037286533787056284381}{
43556142965880123323311949751266331066368}},{\frac{
53999662797066183020074573067574112568763}{
87112285931760246646623899502532662132736}}\right],\\[0.1cm] &
I_{9 }=,\left[{\frac{ 74754707508083115080188235040108949111941}{
87112285931760246646623899502532662132736}},{\frac{
37377353754041557540094117520054474555971}{
43556142965880123323311949751266331066368}}\right],\\[0.1cm] &
I_{10}=,\left[{\frac{ 77964651533856060199939131867310019005089}{
87112285931760246646623899502532662132736}},{\frac{
38982325766928030099969565933655009502545}{
43556142965880123323311949751266331066368}}\right],\\[0.1cm] &
I_{11 }=,\left[1,1\right],\\[0.1cm]
 & I_{12}=,\left[{\frac{
238239449906197613139925190460054741169}{
170141183460469231731687303715884105728}},{\frac{
121978598351973177927641697515548027478529}{
87112285931760246646623899502532662132736}}\right],\\[0.1cm]
 &
I_{13 }=,\left[{\frac{ 61150748376490059476117465769766977687027}{
43556142965880123323311949751266331066368}},{\frac{
122301496752980118952234931539533955374055}{
87112285931760246646623899502532662132736}}\right],\\[0.1cm]
 &
I_{14 }=,\left[{\frac{ 31378001665582463990512894476418249863733}{
21778071482940061661655974875633165533184}},{\frac{
125512006662329855962051577905672999454933}{
87112285931760246646623899502532662132736}}\right],\\[0.1cm] &
I_{15 }=,\left[{\frac{ 173371830902734537060512513441447388046677}{
87112285931760246646623899502532662132736}},{\frac{
86685915451367268530256256720723694023339}{
43556142965880123323311949751266331066368}}\right],\\[0.1cm] &
I_{16 }=,\left[{\frac{ 43530335013270399357309383085115875521389}{
21778071482940061661655974875633165533184}},{\frac{
174121340053081597429237532340463502085557}{
87112285931760246646623899502532662132736}}\right],\\[0.1cm] &
I_{17 }=,\left[2,2\right],\\[0.1cm] &
I_{18 }=,\left[{\frac{ 43581896065820670510185474007765782536433}{
21778071482940061661655974875633165533184}},{\frac{
174327584263282682040741896031063130145733}{
87112285931760246646623899502532662132736}}\right],\\[0.1cm] &
I_{19 }=,\left[{\frac{ 175074047096038860833682791052723868617523}{
87112285931760246646623899502532662132736}},{\frac{
43768511774009715208420697763180967154381}{
21778071482940061661655974875633165533184}}\right],\\[0.1cm] &
I_{20 }=,\left[{\frac{ 207981224135133170727752198124246427433217}{
87112285931760246646623899502532662132736}},{\frac{
103990612067566585363876099062123213716609}{
43556142965880123323311949751266331066368}}\right],\\[0.1cm] &
I_{21 }=,\left[{\frac{ 53389351149121164414177573637629460658781}{
21778071482940061661655974875633165533184}},{\frac{
213557404596484657656710294550517842635125}{
87112285931760246646623899502532662132736}}\right],\\[0.1cm]
\end{align*}
} {\tiny
\begin{align*}
 &
I_{22 }=,\left[{\frac{ 213578718968144293553725465775995745176425}{
87112285931760246646623899502532662132736}},{\frac{
106789359484072146776862732887997872588213}{
43556142965880123323311949751266331066368}}\right],\\[0.1cm]
 &
I_{23}=,\left[{\frac{ 59766163097780263153587646066230635923515}{
21778071482940061661655974875633165533184}},{\frac{
239064652391121052614350584264922543694061}{
87112285931760246646623899502532662132736}}\right],\\[0.1cm] &
 I_{24}=,\left[{\frac{ 253954729472389360941133632461626408129629}{
87112285931760246646623899502532662132736}},{\frac{
126977364736194680470566816230813204064815}{
43556142965880123323311949751266331066368}}\right],\\[0.1cm] &
I_{25}=,\left[{\frac{ 243357617176921003857605925220978518350987}{
43556142965880123323311949751266331066368}},{\frac{
486715234353842007715211850441957036701975}{
87112285931760246646623899502532662132736}}\right],\\[0.1cm] &
I_{26}=,\left[{\frac{ 487996318307203475975406100394067486420489}{
87112285931760246646623899502532662132736}},{\frac{
243998159153601737987703050197033743210245}{
43556142965880123323311949751266331066368}}\right],\\[0.1cm] &
I_{27}=,\left[{\frac{ 854083577046230668182843971813839810572423}{
87112285931760246646623899502532662132736}},{\frac{
106760447130778833522855496476729976321553}{
10889035741470030830827987437816582766592}}\right],\\[0.1cm] &
I_{28}=,\left[{\frac{ 854169231180689628527098761743375311769429}{
87112285931760246646623899502532662132736}},{\frac{
427084615590344814263549380871687655884715}{
43556142965880123323311949751266331066368}}\right],\\[0.1cm] &
I_{29}=,\left[{\frac{ 562325002245556480093549131285848700772183}{
21778071482940061661655974875633165533184}},{\frac{
2249300008982225920374196525143394803088733}{
87112285931760246646623899502532662132736}}\right],\\[0.1cm] &
I_{30}=,\left[{\frac{ 2252874711653478373358567435415938516285773}{
87112285931760246646623899502532662132736}},{\frac{
1126437355826739186679283717707969258142887}{
43556142965880123323311949751266331066368}}\right],\\[0.1cm] &
I_{31}=,\left[{\frac{ 10605493671441752150124990095059709232394013}{
87112285931760246646623899502532662132736}},{\frac{
5302746835720876075062495047529854616197007}{
43556142965880123323311949751266331066368}}\right],\\[0.1cm] &
I_{32}=,\left[{\frac{ 10620731390948027186722372703422173139084449}{
87112285931760246646623899502532662132736}},{\frac{
5310365695474013593361186351711086569542225}{
43556142965880123323311949751266331066368}}\right].
\end{align*}
}

The 11 intervals of maximum length $10^{-40}$ containing in
isolation the positive roots of the polynomial $Q(y)$ are: {\tiny
\begin{align*}
&J_1=\left[{\frac{34099774584992488192779339819330125093373}{
87112285931760246646623899502532662132736}},{\frac{
68199549169984976385558679638660250186747}{
174224571863520493293247799005065324265472}}\right],\\[0.1cm]
&J_2=\left[{\frac{ 76201438867882234589322968824595491762585}{
174224571863520493293247799005065324265472}},{\frac{
38100719433941117294661484412297745881293}{
87112285931760246646623899502532662132736}}\right],\\[0.1cm]
&J_3=\left[{\frac{ 7126628029028087048991427050514157728201}{
10889035741470030830827987437816582766592}},{\frac{
57013024232224696391931416404113261825609}{
87112285931760246646623899502532662132736}}\right],\\[0.1cm]
& J_4=\left[{\frac{ 20353365730967257062570091063747267693705}{
21778071482940061661655974875633165533184}},{\frac{
81413462923869028250280364254989070774821}{
87112285931760246646623899502532662132736}}\right],\\[0.1cm]
&J_5=\left[{\frac{ 20536193605572751128435745571834887603857}{
21778071482940061661655974875633165533184}},{\frac{
82144774422291004513742982287339550415429}{
87112285931760246646623899502532662132736}}\right],\\[0.1cm]
& J_6=\left[{\frac{ 5748340141743633609606285424423240727937}{
5444517870735015415413993718908291383296}},{\frac{
91973442267898137753700566790771851646993}{
87112285931760246646623899502532662132736}}\right],\\[0.1cm]
&J_7=\left[{\frac{ 103898657804950550140822640990534612319235}{
87112285931760246646623899502532662132736}},{\frac{
25974664451237637535205660247633653079809}{
21778071482940061661655974875633165533184}}\right],\\[0.1cm]
&J_8=\left[2,2\right],\\[0.1cm]
&J_9=\left[{\frac{ 193452806356507885400952564030524665947305}{
87112285931760246646623899502532662132736}},{\frac{
96726403178253942700476282015262332973653}{
43556142965880123323311949751266331066368}}\right],\\[0.1cm]
&J_{10}=\left[{\frac{ 225432173349181613519899061972772631001865}{
87112285931760246646623899502532662132736}},{\frac{
112716086674590806759949530986386315500933}{
43556142965880123323311949751266331066368}}\right],\\[0.1cm]
&J_{11}=\left[{\frac{ 237341193967033134227594463359390906202991}{
87112285931760246646623899502532662132736}},{\frac{
14833824622939570889224653959961931637687}{
5444517870735015415413993718908291383296}}\right].
\end{align*}
}


\vspace{-0.2cm}
\newpage

\begin{thebibliography}{12}


\bibitem{AB}  \newblock Y.~Avishai, D.~Berend. \newblock {Transmission through a Thue-Morse chain},
\newblock \emph{Phys. Rev. B.} \textbf{45} (1992), 2717--2724.


\bibitem{BGLL}   \newblock F.~Balibrea, J.L. Garc\'{\i}a Guirao, M.~Lampart,  J.~Llibre.
  \newblock{Dynamics of a Lotka-Volterra map}, \newblock \emph{Fundamenta Mathematicae} \textbf{191} (2006), 265--279.

\bibitem{BL}   \newblock J.~Bernat, J.~Llibre.  \newblock{Counterexample to Kalman and
Markus-Yamabe conjectures in dimension larger than 3}, \newblock
\emph{Dynam. Contin. Discrete Impuls. Systems} \textbf{2} (1996),
337--379.


\bibitem{CLN94}   \newblock J.~Casasayas, J.~Llibre, A.~Nunes. \newblock{Central configurations of the planar $1+n$-body problem}.
\newblock \emph{Celestial Mech. Dynam. Astronom.} \textbf{60} (1994), 273--288.


\bibitem{CGM99}   \newblock A.~Cima, A.~Gasull, F.~Ma\~{n}osas.  \newblock{The discrete
Markus-Yamabe problem}, \newblock \emph{Nonlinear Anal. TMA}
\textbf{35} (1999), 343--354.



\bibitem{CGM14}  \newblock A.~Cima, A.~Gasull, F.~Ma\~{n}osas. \newblock{On the global asymptotic stability of
difference equations satisfying a Markus-Yamabe condition},
\newblock \emph{Publ. Mat.} Extra vol. (2014), 167--178.

\bibitem{CEGHM}   \newblock  A.~Cima,  A.~van den Essen, A.~Gasull, E.~Hubbers, F.~Ma\~{n}osas. \newblock{A polynomial counterexample to the
Markus-Yamabe conjecture}, \newblock \emph{Adv. Math.} \textbf{131}
(1997), 453--457.


\bibitem{CLO}   \newblock J.~M.~Cors, J.~Llibre, M.~Oll\'{e}.
\newblock{Central configurations of the planar coorbital satellite
problem}, \newblock \emph{Celestial Mech. Dynam. Astronom.}
\textbf{89} (2004), 319--342.


\bibitem{DRRS}  \newblock A.~Dickenstein, J.~M.~Rojas, K.~ Rusek, J.~Shih.  \newblock{Extremal Real Algebraic Geometry and
$\mathcal{A}$-Discriminants.} \newblock \emph{Moscow Math. Journal}
\textbf{7} (2007), 425--452.

\bibitem{ED}   \newblock G.H.~Erjaee, F.M.~Dannan.
\newblock{Stability analysis of periodic solutions to the nonstandard
discretized model of the Lotka-Volterra predator-prey system},
\newblock \emph{Int. J. Bifurcation and Chaos} \textbf{14} (2004), 4301--4308.

\bibitem{FPT12}   \newblock E.~Freire, E.~Ponce, F.~Torres. \newblock{The discontinuous matching of
two planar linear foci can have three nested crossing limit cycles},
\newblock \emph{Publ. Mat.} Extra vol. (2014), 221--253.

\bibitem{GGG}   \newblock J.~Garc\'\i a-Salda\~na, A.~Gasull, H.~Giacomini. \newblock{Bifurcation
diagram and stability for a one--parameter family of planar vector
fields}, \newblock \emph{J. Math. Anal. Appl.} \textbf{413} (2014),
321--342.


\bibitem{GGG2}  \newblock
J.~Garc\'\i a-Salda\~na, A.~Gasull, H.~Giacomini. \newblock{
Bifurcation values for a familiy of planar vector fields of degree
five.} \newblock \emph{Discrete Contin. Dyn. Syst.}, \textbf{35}
(2015), 669--701.

\bibitem{GasLlorMan2018}  \newblock A.~Gasull, M.~Llorens, V.~Ma\~{n}osa.
\newblock{Periodic points of a Landen transformation}, \newblock \emph{Commun.
Nonlinear Sci. Numer. Simulat.} \textbf{64} (2018) 232--245.


\bibitem{GD}   \newblock A.~Granas, J.~Dugundji.
\emph{Fixed point theory.} Springer, New York 2003.


\bibitem{G} \newblock  C.~Gutierrez. \newblock A solution to the bidimensional global asymptotic
stability conjecture, \newblock \emph{Ann. Inst. H. Poincar\'e Anal.
Non Lin\'eaire} \textbf{12} (1995), 627--671.



\bibitem{H}   \newblock B.~Haas.  \newblock{A simple counterexample to Kouchnirenko's conjecture},
\newblock \emph{Beitr\"age zur Algebra und Geometrie} \textbf{43} (2002), 1--8.




\bibitem{HY}   \newblock S.M.~Huan, X.S.~Yang. \newblock{On the number of limit cycles in general
 planar piecewise linear systems}, \newblock \emph{Discrete Contin. Dyn. Syst.} \textbf{32} (2012), 2147--2164.

\bibitem{IK}   \newblock  E.~Isaacson, H.B.~Keller. \emph{Analysis of numerical
methods.}
Dover Publications, New York 1994.



\bibitem{Kh}  \newblock A.G.~Khovanski\u{\i} \newblock{On a class of systems of transcendental equations}
 \newblock \emph{Doklady Akad. Nauk. SSSR}, \textbf{255} (1980), 804--807; \newblock \emph{Soviet Math. Dokl.} \textbf{22} (1980) 762--765.


\bibitem{K}   \newblock W.~Kulpa. \newblock{The  Poincar\'{e}-Miranda Theorem},  \newblock \emph{Amer.
Math. Month.} \textbf{104} (1997), 545--550.


\bibitem{LS}   \newblock J.P.~La Salle. \emph{The stability of Dynamical Systems}, CBMS-NSF
Regional Conference Series in Applied Math. \textbf{25}, SIAM 1976,
(2nd printing 1993).


\bibitem{LRW}   \newblock T.-Y.~Li, J.M.~Rojas, X.~Wang. \newblock{Counting real connected
 components of trinomial curve intersections and $m$-nomial hypersurfaces}, \newblock \emph{Discrete Comput. Geom.} \textbf{30} (2003), 379--414.


\bibitem{L}  \newblock J.~Llibre.
\newblock{On the central configurations of the $n$-body problem},
\newblock Preprint. November, 2017.

\bibitem{LP}   \newblock J.~Llibre, E.~Ponce. \newblock{Three nested limit cycles in
 discontinuous piecewise linear differential systems with two zones}, \newblock \emph{Dynam. Contin. Discrete Impuls. Systems}
 \textbf{19} (2012), 325--335.


\bibitem{Mali}  P.~Mali\u{c}k\'y. \newblock{Interior periodic points of a Lotka-Volterra map},
\newblock \emph{J. Difference Eq. Appl.} \textbf{18} (2012), 553--567.



\bibitem{MY}  \newblock L.~Markus, H.~Yamabe. \newblock{Global stability criteria for differential
systems}, \newblock \emph{Osaka Math. Journal} \textbf{12}  (1960),
305--317.

\bibitem{Miranda}  \newblock C.~Miranda. \newblock{Un'osservazione su un teorema di Brouwer}, \newblock \emph{Boll. Unione Mat. Ital.}
\textbf{3} (1940), 527--527.



\bibitem{Poinc1}  \newblock H.~Poincar\'{e}. \newblock{Sur certaines solutions
particulieres du probl\'{e}me des trois corps}, \newblock \emph{C. R.
Acad. Sci. Paris} \textbf{97} (1883), 251--252; and \emph{Bull.
Astronomique} \textbf{1} (1884), 63--74.


\bibitem{Poinc3}   \newblock H.~Poincar\'{e}. \newblock{Sur les courbes d\'{e}finies
par une \'{e}quation diff\'{e}rentielle IV}, \newblock \emph{J. Math. Pures
Appl.} \textbf{85} (1886), 151--217.

\bibitem{S}  \newblock A.N.~Sharkovski\u{\i}. \newblock{Low dimensional dynamics}, Tagungsbericht 20/1993,
\newblock \emph{Proceedings of Mathematisches Forschungsinstitut Oberwolfach}, 1993,
17.

\bibitem{StB}   \newblock J.~Stoer, R.~Bulirsch. \emph{Introduction to Numerical
Analysis}. Springer, New York 2002.



\bibitem{V}  \newblock M.N.~Vrahatis. \newblock{A short proof and a
Generalization of Miranda's existence Theorem}, \newblock
\emph{Proc. AMS} \textbf{107} (1989), 701--703.



\end{thebibliography}
\end{document}